\documentclass[a4paper,reqno]{amsart}

\usepackage[english]{babel}
\usepackage[utf8]{inputenc}

\usepackage[verbose]{geometry}  %
\usepackage{mathtools,amsmath, amsthm, amssymb}
\usepackage{mathrsfs}

\usepackage{bbm} %

\usepackage{enumerate}

\usepackage[alphabetic]{amsrefs} %

\usepackage{hyperref}
\hypersetup{
    colorlinks,
    citecolor=green,
    linkcolor=red,
    urlcolor=[rgb]{0,0,0.75},
    pdftitle={Quadratic sparse domination and weighted estimates for non-integral square functions},
    pdfauthor={Julian Bailey, Gianmarco Brocchi, Maria Carmen Reguera},
    pdfstartview=FitH,
    pdfdisplaydoctitle=true,
    pdflang=en-GB,
    breaklinks=true
  }
  \usepackage[noabbrev,capitalize]{cleveref} %

\usepackage[disable]{todonotes}
\newcommand{\C}{\mathbb{C}}
\newcommand{\R}{\mathbb{R}}
\newcommand{\N}{\mathbb{N}} 
\newcommand{\br}[1]{\left( #1 \right)}
\newcommand{\brs}[1]{\left[ #1 \right]}
\newcommand{\norm}[1]{\left\lVert #1 \right\rVert}
\newcommand{\abs}[1]{\left\lvert #1 \right\rvert}
\newcommand{\lb}[0]{\left\lbrace}
\newcommand{\rb}[0]{\right\rbrace}
\def\Xint#1{\mathchoice
   {\XXint\displaystyle\textstyle{#1}}%
   {\XXint\textstyle\scriptstyle{#1}}%
   {\XXint\scriptstyle\scriptscriptstyle{#1}}%
   {\XXint\scriptscriptstyle\scriptscriptstyle{#1}}%
   \!\int}
\def\XXint#1#2#3{{\setbox0=\hbox{$#1{#2#3}{\int}$}
     \vcenter{\hbox{$#2#3$}}\kern-.5\wd0}}
\def\ddashint{\Xint=}
\def\dashint{\Xint-}

\newcommand{\dyadicM}{ \mathcal{M}^{\mathscr{D}}} %
\newcommand{\1}{ \mathbbm{1}}   %
\newcommand{\D}{ \, \mathrm{d}} %
\renewcommand{\roman}[1]{%
  \textup{\uppercase\expandafter{\romannumeral#1}}%
}
\providecommand{\fint}{\dashint} %

\theoremstyle{plain}
\newtheorem{thm}{Theorem}[section]  %
\newtheorem{prop}[thm]{Proposition}      %
\newtheorem{lem}[thm]{Lemma}
\newtheorem{cor}[thm]{Corollary}

\theoremstyle{plain} %
\newtheorem{lemma}[thm]{Lemma} 
\newtheorem{proposition}[thm]{Proposition} 
\newtheorem{theorem}[thm]{Theorem} 
\newtheorem{corollary}[thm]{Corollary}

\theoremstyle{definition}
\newtheorem{deff}[thm]{Definition}
\newtheorem{assumption}[thm]{Assumption}

\newtheorem{define}[thm]{Definition}

\theoremstyle{remark}
\newtheorem{rmk}[thm]{Remark}
\newtheorem{remark}[thm]{Remark}

\numberwithin{equation}{subsection} %

\title[Quadratic sparse domination for Square Functions]
{Quadratic sparse domination and Weighted Estimates for non-integral Square Functions}

\author[Bailey]{Julian Bailey}
\address{Julian Bailey:
 School of Mathematics\\
  University of Birmingham\\
  Edgbaston\\
  Birmingham\\
  B15 2TT\\
  England}
\email{julianbailey230@gmail.com}

\author[Brocchi]{Gianmarco Brocchi}
\address{
  Gianmarco Brocchi:
  School of Mathematics\\
  University of Birmingham\\
  Edgbaston\\
  Birmingham\\
  B15 2TT\\
  England}
\email{G.Brocchi@pgr.bham.ac.uk}

\author[Reguera]{Maria Carmen Reguera}
\address{
  Maria Carmen Reguera:
  School of Mathematics\\
  University of Birmingham\\
  Edgbaston\\
  Birmingham\\
  B15 2TT\\
  England}
\email{M.Reguera@bham.ac.uk}

\thanks{The first and third authors are supported by
  the UK Engineering and Physical Sciences Research Council (EPSRC) grant EP/P009239/1.
  The second author is supported by
  the (EPSRC) grant  EP/L016516/1
  for the University of Birmingham.} %

\date{}

\begin{document}


\begin{abstract}
  We prove a quadratic sparse domination result 
  for general non-integral square functions $S$. That is, we
  prove an estimate of the form
  \begin{equation*}
    \int_{M} (S f)^{2} g \D{\mu} \le c \sum_{P \in \mathcal{S}}
    \br{\dashint_{5 P} \abs{f}^{p_{0}} \D{\mu}}^{2/p_{0}}
    \br{\dashint_{5 P} \abs{g}^{q_{0}^*}\D{\mu}}^{1/q_{0}^*}
    \abs{P},
  \end{equation*}
  where $q_{0}^{*}$ is the H\"{o}lder conjugate of $q_{0}/2$,
  $M$ is the underlying doubling space and $\mathcal{S}$ is a
  sparse collection of cubes on $M$. Our result will cover both square functions associated with
  divergence form elliptic operators and those associated with the
  Laplace--Beltrami operator.
  This sparse domination allows us to derive optimal norm estimates in the weighted space $L^{p}(w)$.
\end{abstract}

\maketitle

\section{Introduction}
\label{sec:Introduction}
 Recent years have seen a surge of activity in the weighted theory of singular integrals that has resulted in the resolution of some major conjectures such as the $A_2$ Conjecture \cite{MR2912709}, the Muckenhoupt--Wheeden Conjecture \cite{MR2799801} and the resolution of the two weight problem for the Hilbert transform \cites{MR3285858, MR3285857}. Accompanying these achievements is the development of new core techniques such as the representation of singular integrals by dyadic operators \cite{MR2912709} or the sparse domination of singular integrals \cites{MR3521084, MR4007575}. 

The sparse domination of an operator provides, at a glance, a complete
picture of the unweighted and weighted estimates with precise tracking
of the dependence on the weight characteristic.
Its use in harmonic analysis
was introduced by Lerner in \cite{MR2721744}, where a decomposition of
an arbitrary measurable function was obtained in terms of its local
mean oscillations.
It has since been extended to
a great number of different contexts spanning and reproducing a large
portion of harmonic analysis. To call attention to some of the most celebrated results, there have
been articles published
covering the domination of Calder\'{o}n--Zygmund singular integral operators
\cites{MR3085756,MR3625108,MR3484688,MR4007575}, multilinear singular integrals, rough
singular integrals, variational Carleson, Bochner--Riesz multipliers,
Walsh--Fourier multipliers, spherical maximal function and also the $T1$ sparse
domination of singular integrals. For more details on these and other
applications we refer the reader to Section 8 of the survey paper
\cite{zbMATH07215908} and the references therein. In this article we
are interested in the sparse domination of square function operators.

The sparse domination of classical square function operators was
first considered in \cite{MR2770437}. In this article it was
discovered that in order to obtain sharp weighted estimates for square
functions from a
sparse domination result, the sparse techniques applied to singular integral operators had to be adjusted to account for the quadratic
nature of the square function. Thus, instead of a ``linear'' sparse
domination result, one must aim for a stronger ``quadratic'' sparse domination theorem. This idea was also explored in \cite{Brocchi} where a
quadratic result with minimal $T1$-type assumptions is proved. Similar
ideas are also investigated in the work of Lorist \cite{Lorist}, where
sparse domination is proved for general vector-valued operators.

Since the turn of the century, fuelled by applications to boundary
value problems and the epic contest of ideas
surrounding the Kato conjecture, there has been a sustained and pronounced interest in weighted estimates for non-integral singular
operators that are beyond the realm of Calder\'{o}n--Zygmund
theory. Some of the most prominent examples of which are operators attached
to the divergence form elliptic operator $L = - \mathrm{div} (A
\nabla)$, where $A$ is bounded and elliptic with complex coefficients.
For instance, neither the Riesz transforms $\nabla L^{-\frac{1}{2}}$
nor the constituent operators $\{\sqrt{t} \nabla e^{-t L}\}_{t>0}$ of the square function
  \begin{equation*}
    G_{L} f = \br{\int^{\infty}_{0} \abs{\sqrt{t} \nabla e^{-t L} f}^{2} \frac{\D{t}}t}^{1/2}
  \end{equation*}
possess integral kernels that satisfy any meaningful estimates and, as
such, are deserving of the title ``non-integral''. As a result of this
characteristic, and in contrast to the classical setting of the Laplacian operator $-\Delta$, these operators will fail to be bounded on $L^{p}(\R^{n})$ for $p$ in the entire
interval range $(1,\infty)$. Instead, as proved in
\cite{auscher2007necessary}, boundedness will occur if and only if $p$
is contained within a restricted
subinterval of $(1,\infty)$ that will depend on the perturbation $A$. Similarly, for boundedness on the weighted space
$L^{p}(w)$, one must also consider a restricted range of $p \in (1,\infty)$. For a detailed
investigation into such results the reader is referred to the seminal
series of papers by P. Auscher and J. M. Martell,
\cites{AuscherMartellIGeneral,AuscherMartellII,AuscherMartellIII,AuscherMartellIV}.

The sparse domination methods developed for
Calder\'{o}n--Zygmund operators in \cites{MR3085756,MR3625108,MR3484688} 
automatically imply boundedness on $L^{p}(\R^{n})$ for $p$ in the full
range $(1,\infty)$. It then follows that the classical sparse domination is
particularly ill-suited to non-integral singular operators.
In the article \cite{bernicot2016sharp}, the authors F. Bernicot,
D. Frey and S. Petermichl introduced a linear sparse domination framework
that was adapted to non-integral singular operators in the sense that
the sparse object dominating the operator would only be bounded on a
restricted range.  This linear sparse domination allowed for sharp weighted estimates to be
produced for a wide range of operators associated with
$L$ that included the Riesz transforms $\nabla
L^{-\frac{1}{2}}$.

As stated earlier for the classical setting of the Laplacian,
the linear sparse domination in \cite{bernicot2016sharp}
does not imply the best weighted bounds for square functions
for $p>2$.
The ultimate objective of this article is thus
to prove a quadratic sparse domination theorem for non-integral
square functions. This, in turn, will yield weighted estimates for
square functions similar to $G_{L}$. These weighted estimates will be optimal in the sense that they will
reproduce the sharp weighted estimates of \cite{MR2770437} for
$G_{-\Delta}$. Similarly, they will also reproduce optimal weighted
estimates for $G_{L}$ when $L = - \mathrm{div}(A \nabla)$ and $A$ is
real-valued, a result that was first proved  by T. A. Bui and X. Duong in
\cite{MR4058541}.

Motivated by finding a uniform setting that will include several examples of square functions,
we consider the following general framework.
The underlying space $(M,d,\mu)$ is a locally compact separable metric space $(M,d)$ equipped
with a Borel measure $\mu$ that is finite on compact sets and strictly
positive on any non-empty open set. For a subset $B \subset M$ of non-vanishing and finite measure,
we denote $\lvert B \rvert \coloneqq \mu(B)$.

The measure $\mu$ will be assumed to satisfy the doubling property,
\begin{equation}
  \label{eqtn:Doubling1}
  \abs{B(x,2 r)} \lesssim \abs{B(x,r)}
\end{equation}
for all $x \in M$ and $r > 0$, where $B(x,s)$ denotes the ball of radius $s > 0$ centered at a
point $x \in M$ and  $X \lesssim Y$ will be used throughout the paper to signify that
there exists a constant $C>0$ such that $X \le C Y$.

There will then exist some $\nu > 0$ for which
\begin{equation}
  \label{eqtn:Doubling2}
  \abs{B(x,r)} \lesssim \br{\frac{r}{s}}^{\nu} \abs{B(x,s)} \qquad
  \forall \, x \in M, \, r \geq s > 0.
\end{equation}

It will be assumed that there exists some non-decreasing function
$\varphi : (0,\infty) \rightarrow (0,\infty)$ with $\varphi(1) = 1$ for which
\begin{equation}
  \label{eqtn:Growth}
\abs{B(x,r)} \simeq \varphi\br{\frac{r}{s}} \abs{B(x,s)}
\end{equation}
for all $x \in M$ and $r, \, s > 0$, where $X \simeq Y$ means that both $X \lesssim Y$ and $Y \lesssim X$ hold. This technical condition has been
imposed in order to prove boundedness of a certain maximal operator
that is essential to our proof. This point will be elaborated upon
further in Remark
\ref{rmk:Assumptions} and Section \ref{sec:Maximal}.

Let $L$ be an unbounded operator on $L^{2}(M,\mu)$ satisfying the
below assumption.
  
\begin{assumption}
  \label{assum:L}
  $L$ is an injective linear operator on $L^{2}(M,\mu)$ with dense domain
  $\mathcal{D}_{2}(L) \subset L^{2}(M,\mu)$. $L$ is 
$\omega$-accretive for some $0 \leq \omega < \pi / 2$ and there
exists some $1 \leq p_{0} < 2 < q_{0} \leq \infty$  and $c > 0$ such that for all
balls $B_{1}, \, B_{2}$ of radius $\sqrt{t}$,
$$
\norm{e^{-t L}}_{L^{p_{0}}(B_{1}) \rightarrow L^{q_{0}}(B_{2})}
\lesssim \abs{B_{1}}^{-\frac{1}{p_{0}}} \abs{B_{2}}^{\frac{1}{q_{0}}}
e^{-c \frac{d(B_{1},B_{2})^{2}}{t}}.
$$
\end{assumption}

 From \cref{assum:L}, it follows that $L$ is a maximal accretive
operator on $L^{2}(M,\mu)$, $L$ possesses a bounded
holomorphic functional calculus on $L^{2}(M,\mu)$ and $-L$ is the
generator of an analytic semigroup $(e^{-t L})_{t > 0}$ on
$L^{2}(M,\mu)$.

Throughout this article,
we consider square function operators
associated with $L$. These will be defined to be operators $S$ that satisfy the
following set of assumptions.

\begin{assumption}
  \label{assum:S}
  \begin{enumerate}
  \item[(a)] The operator $S$ is sublinear and bounded on  $L^{2}(M,\mu)$.   
  \item[(b)] \emph{(Off-diagonal estimates for the constituent operators)}
    The operator $S$ is of the form
    $$
    Sf(x) := \br{\int^{\infty}_{0} \abs{\mathcal{Q}_{t}f(x)}^{2} \, \frac{\D{t}}{t}}^{\frac{1}{2}},
    $$
    where $\lb \mathcal{Q}_{t} \rb_{t > 0}$ is a collection of bounded operators
    on $L^{2}(M,\mu)$ which satisfy the property that for balls $B_{1},
    \, B_{2}$ of radius $\sqrt{t}$,
    $$
    \norm{\mathcal{Q}_{t}}_{L^{p_{0}}(B_{1}) \rightarrow L^{q_{0}}(B_{2})}
    \lesssim \abs{B_{1}}^{-\frac{1}{p_{0}}} \abs{B_{2}}^{\frac{1}{q_{0}}}
    \br{1 + \frac{d(B_{1},B_{2})^{2}}{t}}^{-(\nu + 1)}.
    $$

  \item[(c)] \emph{(Cancellation with respect to $L$)}
    There exists $A_{0} > 0$ and $N_{0} \in \N$ such that for all
    integers $N \geq N_{0}$,
    $$
    \mathcal{Q}_{s} (t L)^{N} e^{-t L} = \frac{s^{A_{0}} t^{N}}{(s +
      t)^{A_{0} + N}} \Theta_{s + t}^{(N)},
    $$
    where $\lbrace \Theta_{r}^{(N)} \rbrace_{r > 0}$ is a collection of bounded
    operators on $L^{2}(M,\mu)$ that satisfies off-diagonal estimates
    at all scales
    in the sense that
    \begin{equation*}
      \big\lVert \Theta_{r}^{(N)} \big\rVert_{L^{p_{0}}(B_{1}) \rightarrow L^{q_{0}}(B_{2})}
      \lesssim \abs{B_{1,\sqrt{r}}}^{-\frac{1}{p_{0}}} \abs{B_{2,\sqrt{r}}}^{\frac{1}{q_{0}}} \br{1 + \frac{d(B_{1},B_{2})^{2}}{r}}^{-
        \frac{\nu + 1}{2}}
    \end{equation*}    
    for all balls $B_{1}, \, B_{2}  \subset M$ and $r > 0$, where
    $B_{i,\sqrt{r}} := (\sqrt{r}/r(B_{i})) B_{i}$ for $i = 1, \, 2$ and  for a ball $B = B(x,r)$
and $t > 0$, we will use the notation $t B$ to represent the
$t$-dilate of $B$, $t B \coloneqq B(x,t r)$.

\item[(d)] \emph{(Cotlar type inequality)} There exists an exponent $p_{1} \in [p_{0},2)$ such that
  for all $x \in M$ and $r > 0$
  $$
\br{\dashint_{B(x,r)} \abs{S e^{-r^{2} L} f}^{q_{0}} \,
  \D{\mu}}^{\frac{1}{q_{0}}} \lesssim \inf_{y \in B(x,r)}
\mathcal{M}_{p_{1}}(Sf)(y) + \inf_{y \in B(x,r)} \mathcal{M}_{p_{1}}(f)(y),
$$
where we define
 $\dashint_{B}f \D{\mu} \coloneqq |B|^{-1} \int_{B}f \D{\mu}$ for $f\in L^1_{\mathrm{loc}}(M,\mu)$ and 
we denote by $\mathcal{M}$ the uncentered Hardy--Littlewood maximal function
and $\mathcal{M}_pf \coloneqq (\mathcal{M} \lvert f\rvert^p
)^{1/p}$ for any $p\ge 1$. 
\end{enumerate}
\end{assumption}

\begin{rmk}
  \label{rmk:Assumptions}
  As our work is intended to build upon the article
  \cite{bernicot2016sharp}, it will be instructive to compare our
  assumptions with the hypotheses of \cite{bernicot2016sharp}. In both
  our article and \cite{bernicot2016sharp}, the assumptions imposed upon
  the underlying operator $L$ are identical. For the operator $S$, we
  have also assumed $L^{2}$-boundedness and a Cotlar
  type inequality. However, we have included the additional assumption that $S$ is of
  the form of a square function composed of operators $\mathcal{Q}_{t}$
  that satisfy off-diagonal bounds. Also, the cancellative condition of
  $S$ with
  respect to $L$, Assumption 1.1(b) of \cite{bernicot2016sharp}, has
  instead been replaced by a cancellative condition of the constituent operators
  $\mathcal{Q}_{t}$.

  In Section \ref{sec:Maximal}, using the growth
  condition imposed upon our metric space \eqref{eqtn:Growth}, it
  will be proved that the assumed cancellative condition for the
  $\mathcal{Q}_{t}$ operators does in fact imply the cancellative
  condition of $S$ with respect to $L$. This allows us to deduce that
  the operators under consideration in our article
  are a restricted subclass of the operators considered by
  \cite{bernicot2016sharp}. Indeed, the additional growth condition of our
  metric space \eqref{eqtn:Growth} has been assumed with the sole
  purpose of ensuring that we are working strictly within the setting of
  \cite{bernicot2016sharp}. This will allow us to utilise some of the
  intermediary results from \cite{bernicot2016sharp} 
  without having to reprove them under a different cancellation condition. This will be of particular use to
  us in Section \ref{sec:Maximal} when we come to prove
  the boundedness of a certain maximal function operator that is essential to
  our proof.
\end{rmk}

\begin{rmk}
One does not have to search for long before encountering examples of square function operators that
satisfy the previous set of assumptions. For instance, the square
functions associated with an elliptic operator $L = - \mathrm{div} A \nabla$, denoted by $g_{L}$
and $G_{L}$ and discussed in \cite{AuscherMartellIII}, and square
functions associated with the Laplace--Beltrami $\Delta$ satisfy
the above conditions. We discuss these examples in detail in Section 3.
\end{rmk}

In order to make sense of the concept of sparse domination and precisely
state our main theorem we need to define the notion of a sparse
family of cubes.
We consider a system of dyadic cubes $\mathscr{D}$ on the metric space $(M,d)$.
\begin{define}%
  A collection of dyadic cubes $\mathcal{S} \subseteq \mathscr{D}$ is
  $\frac12$-sparse if there exists a disjoint collection of sets $\{
  F_P \,:\, P \in\mathcal{S}\}$ such that for every $P \in
  \mathcal{S}$ we have $F_{P} \subset P$ and $\lvert F_P\rvert > \frac12 \lvert P\rvert$.
\end{define}

\begin{theorem}\label{t.main}
  Let $p_0<2<q_0$ and consider operators $L$ and $S$ that satisfy
  Assumptions \ref{assum:L} and \ref{assum:S} for this choice of
  exponents.
  For any $f$ and $g$ in $C^\infty_c(M)$
  there exists a sparse family $\mathcal{S} \subseteq \mathscr{D}$ 
  such that
  \begin{equation}\label{e.main}
     \int_{M} (Sf)^2  g \,\D{\mu}
     \le c \sum_{P\in\mathcal{S}} \left( \dashint_{5P} \lvert f \rvert^{p_0} \D{\mu} \right)^{2/p_0}  \left(\dashint_{5P} \lvert g\rvert^{q_0^*} \D{\mu} \right)^{1/q_0^*} |P|,
  \end{equation}
  where $q*:=\left (\frac{q}{2} \right)'$ is the H\"{o}lder conjugate
  of $\frac{q}{2}$, and $c$ is a positive constant independent of $f$ and $g$.
\end{theorem}

The right hand side of \eqref{e.main} is the sparse form natural to
the square function. We observe that the bilinear sparse form obtained
differs from the linear sparse domination results where the $L^{q_{0}'}$ average of $g$
is used instead (c.f. \cite{bernicot2016sharp}). This is due to the non-linear nature of the problem at hand. Analogous sparse forms appear when controlling vector-valued operators, as seen in the work of Lorist \cite{Lorist}. In fact, as the operators we consider satisfy the hypotheses from
\cite{bernicot2016sharp}, it follows that 
{\cite[Thm.~5.7]{bernicot2016sharp}} will be valid for
$S$. This result states that for any $f$ and $g$ in $C^\infty_c(M)$
  there exists a sparse family $\mathcal{S} \subseteq \mathscr{D}$ 
  for which
\begin{equation}
  \label{eqtn:BFP}
  \abs{\int_{M} S f \cdot g \, \D{\mu}}
  \lesssim \sum_{P \in \mathcal{S}} \br{\dashint_{5 P} \abs{f}^{p_{0}}
  \D{\mu}}^{1/p_{0}} \br{\dashint_{5 P} \abs{g}^{q_{0}'} \D{\mu}}^{1/q_{0}'} \abs{P}.
\end{equation}

The essence of our sparse domination result is that, under the
additional square function hypotheses assumed above, the previous
sparse bound can be improved to a quadratic sparse domination bound
that is uniquely suited to square function operators. 

Our proof strategy requires the weak boundedness at the endpoint of
a``grand maximal function'' operator associated with the square function. This
strategy is an adaptation of Lerner's work on singular integrals \cite{MR3484688}
to our setting, which itself is an elaboration of
Lacey's elementary proof from \cite{MR3625108}.  The weak-type boundedness
of our grand maximal operator will be obtained by demonstrating
that our operator is pointwise controlled from above by a related
maximal operator that was
introduced in \cite{bernicot2016sharp}. The weak boundedness of this
alternative grand maximal operator was proved in
\cite{bernicot2016sharp} under their setting. Since, as will be shown
in Section \ref{sec:Maximal}, we are working strictly within their
setting, this will then allow us to conclude that our grand maximal
operator is also weakly bounded at the endpoint.

Next we give an account of the weighted estimates
that we obtain for our square functions via the sparse domination (Theorem \ref{t.main}).
It is understood that if the operator at hand maps $L^p$ to $L^p$ for a restricted range of exponents $p$,
the relevant classes of weights will involve the intersection
of Muckenhoupt and Reverse H\"older weights \cite{AuscherMartellIGeneral}.
We define them precisely.

A weight $w$ is a positive locally integrable function. We say that a weight $w$ is in the Muckenhoupt $A_p$ class for $1<p<\infty$ and we denote it by $w\in A_p$ if and only if
\begin{equation*}
  [w]_{A_p}:=\sup_{Q \,\,\text{cube }} \dashint_{Q} w \D{\mu}
  \left(\dashint_{Q} w^{1 - p'} \D{\mu} \right)^{p-1}<\infty,
\end{equation*}
where $p'=p/(p-1)$ is the dual exponent of $p$.
We say that a weight $w$ belongs to the reverse H\"older class $RH_p$ for $p>1$ if 
\begin{equation*}
  [w]_{RH_p}:= \sup_{Q \,\,\text{cube }} \left(\dashint_{Q} w^p \D{\mu} \right)^{1/p} \left( \dashint_{Q} w \D{\mu} \right)^{-1}<\infty.
\end{equation*}

We can now state our second result.
\begin{theorem} \label{thm:weight_bounds}
  Fix $p_0<2<q_0$. For any sparse family
  $\mathcal{S} \subset \mathscr{D}$,
  functions $f,\,g \in L^{1}_{\operatorname{loc}}(\D{\mu})$,
  $p \in (2,q_{0})$ and weight $w\in A_{\frac{p}{p_{0}}}\cap RH_{\left(\frac{q_{0}}{p}\right)^{'}}$ we have
  \begin{equation}\label{e.secondmain}
    \sum_{P \in \mathcal{S}} \left(\dashint_{5 P} \abs{f}^{p_0}\D{\mu}\right)^{2/p_0} \left(\dashint_{5 P}  \abs{g}^{q_{0}^{*}} \D{\mu} \right)^{1/q_{0}^{*}} \abs{P} 
    \leq C_0 \Big(\brs{w}_{A_{\frac{p}{p_{0}}}} \cdot \brs{w}_{RH_{\big(\frac{q_{0}}{p}\big)'}}\Big)^{2\gamma(p)} \norm{f}_{L^{p}(w)}^{2} \norm{g}_{L^{p^{*}}(\sigma)},
  \end{equation}
  where
  \begin{equation*}
    \gamma(p) \coloneqq \max \br{\frac{1}{p - p_{0}}, \br{\frac{q_{0}}{p}}' \frac{1}{2q_{0}^{*}}} \,\,\, \text{ and } \,\,\, \sigma \coloneqq w^{1-p^{*}}.
  \end{equation*}
 The constant $C_0$ is independent of both the weight and the sparse
  collection, and the dependence of this estimate on the weight characteristic is sharp.
\end{theorem}

Expanding further upon the above theorem, the result is sharp in the
sense that the dependence on the weight characteristic $\brs{w}_{A_{p/p_{0}}} \brs{w}_{RH_{(q_{0}/p)'}}$
can be matched at least asymptotically with the right choice of
functions, weights and sparse form. A detailed proof of this sharpness
will be presented in \cref{sec:Sharpness}.
The above theorem, when combined with our other main result, Theorem
\ref{t.main}, allow us to obtain as a corollary the following sharp weighted
result for non-integral square functions. It is important to note that
the combination of Theorems \ref{t.main} and \ref{thm:weight_bounds}
only produces the below weighted bounds for $p \in (2, q_{0})$. The
weighted estimates for the full range $p \in (p_{0},q_{0})$
follows from this on applying a restricted range extrapolation theorem
provided by Auscher and Martell in 
 \cite[Thm.~4.9]{AuscherMartellIGeneral}.

\begin{corollary}
  \label{cor:Weighted}
  Let $p_0<2<q_0$ and consider operators $L$ and $S$ that satisfy
  Assumptions \ref{assum:L} and \ref{assum:S} for this choice of exponents. For $p \in (p_{0},q_{0})$ and  $w \in A_{\frac{p}{p_{0}}} \cap
  RH_{\br{\frac{q_{0}}{p}}'}$ the square function $S$ is bounded on $L^{p}(w)$
  with
  \begin{equation}
    \label{eqtn:Bounded}
    \norm{S}_{L^{p}(w)} \lesssim \br{\brs{w}_{A_{\frac{p}{p_{0}}}} \cdot \brs{w}_{RH_{(\frac{q_{0}}{p})'}}}^{\gamma(p)},
  \end{equation}
  where $\gamma(p)$ is as defined in Theorem \ref{thm:weight_bounds}.
\end{corollary}

The result is sharp in the sense that for certain square functions, such as $S=g_{-\Delta}$  or $S=g_{L}$ when $L = - \mathrm{div} A \nabla$ and $A$ is
real-valued, the operator norm can be attained asymptotically, see \cite{MR2770437} and \cite{MR4058541} respectively. Sharpness can also be deduced from the asymptotic behaviour of the unweighted estimates \cite{FreyBas}. Unfortunately, these asymptotics are not easy to exactly compute for our non-integral square functions. 
However, the estimate \eqref{eqtn:Bounded} implies an upper bound
on the asymptotic behaviour of the unweighted norm $\lVert S \rVert_{L^p \to L^p}$,
see \cref{subsec:asymptotic_behaviour}.
In particular, when such asymptotic behaviour is known to match the upper bound,
the weighted estimates in \cref{cor:Weighted} are sharp.
For example, this is the case when $L = - \Delta$.

\subsection*{Structure of the paper} The paper is distributed as follows.
Section 2 contains some preliminary results that will be of use later in the paper.
Section 3 will discuss the examples that fit the assumptions and that one should keep in mind as references.
The proof of Theorem \ref{t.main} requires us to understand the boundedness properties of a grand maximal operator
associated with the corresponding square functions.
These boundedness properties are included in Section 4.
Section 5 is dedicated to the proof of our main result, Theorem \ref{t.main}.
Section 6 considers weighted estimates for the sparse forms found
in Section 5 and, in particular, proves \cref{thm:weight_bounds}.
Finally, Section 7 is dedicated to the proof of the sharpness of \cref{thm:weight_bounds} when $p>2$.

\vspace{0.1in}

\section{Preliminaries}
\label{sec:Prelims}

In this section we gather some of the main results
about dyadic analysis in measure metric spaces,
off-diagonal estimates for a family of operators,
and properties of Muckenhoupt and reverse Hölder weight classes.

\subsection{Dyadic Analysis on a Doubling Metric Space}
\label{subsec:Dyadic_on_metric_spaces}

We recall some well-known definitions and facts from dyadic harmonic
analysis as written in \cite{bernicot2016sharp}. For detailed
information on the construction of dyadic systems of cubes in doubling
metric spaces, the interested reader is referred to
\cite{HytKairema} and references therein.

\begin{deff} 
 \label{def:DyadicSystem} 
 A dyadic system of cubes in $(M,\mu)$, with parameters $0 < c_{0} \leq
 C_{0} < \infty$ and $\delta \in (0,1)$, is a family of open subsets
 $\br{Q^{l}_{\alpha}}_{\alpha \in \mathcal{A}_{l}, l \in \mathbb{Z}}$ that
 satisfies the following properties:

\begin{enumerate}
\item[$\bullet$] For each $l \in \mathbb{Z}$, there exists a subset $Z_{l}$
  with $\mu(Z_{l}) = 0$ such that
$$
M = \bigsqcup_{\alpha \in \mathcal{A}_{l}} Q^{l}_{\alpha} \bigsqcup Z_{l};
$$
\item[$\bullet$] If $l \geq k$, $\alpha \in \mathcal{A}_{k}$ and
  $\beta \in \mathcal{A}_{l}$ then either $Q^{l}_{\beta} \subseteq
  Q^{k}_{\alpha}$ or $Q^{k}_{\alpha} \cap Q^{l}_{\beta} = \emptyset$;
\item[$\bullet$] For every $l \in \mathbb{Z}$ and $\alpha \in
  \mathcal{A}_{l}$, there exists a point $z^{l}_{\alpha}$ with the
  property that
$$
B(z^{l}_{\alpha}, c_{0} \delta^{l}) \subseteq Q^{l}_{\alpha} \subseteq
B(z^{l}_{\alpha}, C_{0} \delta^{l}).
$$
\end{enumerate}
\end{deff}

The below theorem asserts the existence of adjacent systems of dyadic
cubes for a doubling metric space. For a proof of this result, refer
to \cite{HytKairema} and references therein.

\begin{thm}[{\cite[Thm. 4.1]{HytKairema}}]
 \label{thm:Dyadic} 
 There exists $0 < c_{0} \leq C_{0} < \infty$, $\delta \in (0,1)$,
 finite constants $K = K(c_{0},C_{0},\delta)$ and $C = C(\delta)$,  and a
 finite collection of dyadic systems $\mathscr{D}^{b}$, $b = 1,\cdots, K$,
 that satisfies the following property. For any ball $B = B(x,r)
 \subseteq M$, there exists $b \in \lb 1, \cdots, K \rb$ and $Q \in
 \mathscr{D}^{b}$ such that 
 \begin{equation*}
   B \subseteq Q \quad and \quad \mathrm{diam}(Q) \leq C r.
 \end{equation*}
\end{thm}

From this point forward we fix a dyadic collection
$\mathscr{D} \coloneqq \cup_{b = 1}^{K} \mathscr{D}^{b}$ as in the previous theorem. 
Let $w$ be a weight on $M$.
For $P\in\mathscr{D}$ we denote by $\ell(P)$ the side length of $P$.
The uncentered dyadic maximal function $\dyadicM_{p,w}$ of exponent $p
\in [1,\infty)$ is defined by
\begin{equation*}
  \dyadicM_{p,w} f(x) \coloneqq \sup_{Q\in\mathscr{D}} \left( \frac{1}{w(Q)}\int_{Q} \lvert f(y) \rvert^{p} w(y) \D{y} \right)^{1/p} \1_Q(x),
\end{equation*}
where the notation $\mathbbm{1}_{E}$ is used to denote the
characteristic function of a set $E \subset M$, and $w(E) \coloneqq \int_E w \D{\mu}$.
When $w \equiv 1$, $\dyadicM_{p,w}$ will just be the usual dyadic
maximal function of exponent $p$ and the shorthand notation
$\dyadicM_{p} = \dyadicM_{p,1}$ will be employed. Similarly, we
will also use the notation $\dyadicM_{w} = \dyadicM_{1,w}$.
It is known that  $\dyadicM_{p}$ is of weak-type $(p,p)$ and strong $(q,q)$ for all $q>p$, see \cite{CoifWeiss}.
Moreover, $\dyadicM_w$ is bounded on $L^p(w)$ for all $p \in [1,\infty)$ with a constant independent of the weight, 
\begin{equation}
  \lVert \dyadicM_w f \rVert_{L^p(w)} \le p' \lVert f \rVert_{L^p(w)}.\label{eq:bound_M_w}   
\end{equation}

\subsection{Off-Diagonal Estimates}

In this section, we define three different notions of off-diagonal estimates that
will be used throughout this article. For an extensive and detailed
account of off-diagonal estimates for operator families, the reader is
referred to \cite{AuscherMartellII}. Throughout this section, we will
consider exponents $1 \leq p_{0} < 2 \leq q_{0} \leq \infty$.

\begin{define}[Off-diagonal estimates at scale $\sqrt{t}$]
  \label{def:off-diagonal_ests}
  A family of operators $\{T_t\}_{t>0}$ is said to satisfy $(p_0,q_0)$ off-diagonal estimates
  at scale $\sqrt{t}$
  if for any two balls $B_1,B_2$ of radius $\sqrt{t}$ we have
  \begin{equation*}
    \left(\fint_{B_2} \lvert T_t (f\1_{B_1})\rvert^{q_0} \D{\mu} \right)^{1/q_0} \lesssim \rho\Big(\frac{d(B_1,B_2)}{\sqrt{t}}\Big) \left( \fint_{B_1} \lvert f\rvert^{p_0} \D{\mu} \right)^{1/p_0}
  \end{equation*}
  where $\rho \colon [0,\infty) \to (0,1] $ is a non increasing function such that
  $\rho(0)=1$ and $\lim_{x\to\infty}\lvert x\rvert^a \rho(x)=0$ for some $a \ge 0$.
\end{define}
\begin{remark} Some comments are in order.
  \begin{itemize}
  \item Examples of $\rho$ that we will use are the Gaussian
    $\rho(x) = e^{-c\lvert x\rvert^2}$ and
    $\rho(x) = \langle x \rangle^{-c(\nu+1)}$ where
    $\langle x \rangle^{s} = (1 + \lvert x \rvert^2)^{s/2}$ is the
    Japanese bracket.  For the Gaussian case, the  positive constant $c$ is not relevant and
    may change from line to line.
    See also comments after \cite[Def.~2.1]{AuscherMartellII}.
    For our sparse domination, the choice $\rho(x)=\langle x\rangle^{-2(\nu + 1)}$ suffices.
    
  \item Off-diagonal estimates are stable under composition. That is, if
    $T_{t}$ satisfies $(p_{1},p_{2})$ off-diagonal estimates at scale
    $\sqrt{t}$ and $S_{t}$ satisfies $(p_{2},p_{3})$ off-diagonal
    estimates at scale $\sqrt{t}$ then $S_{t} T_{t}$ will satisfy
    $(p_{1},p_{3})$ off-diagonal estimates at scale $\sqrt{t}$. Again,
    the value of $c$ or $s$ may change.

  \item For $p_{0} \leq p \leq q \leq q_{0}$, H\"{o}lder's inequality
    implies that if an operator family satisfies $(p_{0},q_{0})$
    off-diagonal estimates at scale $\sqrt{t}$ then it will also
    satisfy $(p,q)$ estimates.
    
  \item Off-diagonal estimates for $p\le q$ do not imply $L^p-L^q$ boundness of $T_t$,
    see \cite{AuscherMartellII}.
  \end{itemize}
\end{remark}

In order to apply off-diagonal estimates,
we often need to decompose the support of a function $f$
into finitely overlapping balls with radius to match the scale.

\begin{define}
  We say that a collection of balls $\mathcal{B}$ has finite overlap
  if there exists a finite constant $\Lambda_{\mathcal{B}}$ such that
  \begin{equation*}
    \lVert \sum_{B \in \mathcal{B}} \1_B \rVert_{L^\infty} = \Lambda_{\mathcal{B}} .
  \end{equation*}  
\end{define}

\begin{rmk}\label{rmk:finite_overlap_balls}
  Let $\mathcal{B}$ be a collection of 
  finite overlapping balls
  covering a set $\Omega$.
  Then
  \begin{equation*}
    \sum_{B \in \mathcal{B}} \mu(B) = \int_{\Omega} \sum_{B \in \mathcal{B}} \1_B \D{\mu} \le \Lambda_{\mathcal B} \; \mu(\Omega) .
  \end{equation*}
\end{rmk}

\begin{lemma}\label{lemma:counting_R_balls}
  Let $\Omega \subset M$ be an open set, and 
  let $\mathcal{R}$ be a family of finite overlapping balls, with the same radius, %
  covering $\Omega$. If there exists $m \in \mathbb{N}$ such that
  $m R \supset \Omega$ for all $R \in \mathcal{R}$,
  then for any $f \in L^{p_0}(\Omega)$, $p_0 \ge 1$, we have
  \begin{equation}\label{eq:counting_R_balls}
    \sum_{R\in\mathcal{R}} \left( \fint_R \lvert f\rvert^{p_0} \D{\mu} \right)^{1/p_0} \lesssim m^\nu \left( \fint_{\Omega} \lvert f\rvert^{p_0} \D{\mu} \right)^{1/p_0} .
  \end{equation}  
\end{lemma}
\begin{proof}
  For $p_0 >1$,
  H\"{o}lder's inequality implies that
  \begin{align*}
    \sum_{R \in \mathcal{R}} \br{\dashint_{R} \abs{f}^{p_{0}} \D{\mu}}^{\frac{1}{p_{0}}}
    &\leq \br{\sup_{R \in \mathcal{R}} \frac{1}{\abs{R}}}^{\frac{1}{p_{0}}}
      \br{\sum_{R \in \mathcal{R}} \int_{R} \abs{f}^{p_{0}} \D{\mu}}^{\frac{1}{p_{0}}} \br{\sum_{R \in \mathcal{R}}1}^{\frac{1}{p_{0}'}} \\
    &= \br{\sup_{R \in \mathcal{R}} \frac{\lvert \Omega\rvert}{\abs{R}}}^{\frac{1}{p_{0}}}
      \br{\dashint_{\Omega} \abs{f}^{p_{0}}\D{\mu}}^{\frac{1}{p_{0}}} \br{\# \mathcal{R}}^{\frac{1}{p_{0}'}}.
  \end{align*}
  Since $m R \supset \Omega$ for all $R \in \mathcal{R}$,
  the doubling property implies that
  \begin{equation*}
    \br{\sup_{R \in \mathcal{R}} \frac{\lvert \Omega\rvert}{\abs{R}}}^{\frac{1}{p_{0}}}
    \br{\# \mathcal{R}}^{\frac{1}{p_{0}'}}
    \lesssim \sup_{R \in \mathcal{R}}\frac{\lvert m R\rvert}{\abs{R}} \lesssim m^\nu.
  \end{equation*}
  The case $p_{0} = 1$ is even simpler since it does not require the
  use of H\"{o}lder's inequality nor an estimate on the cardinality
  $\# \mathcal{R}$.
\end{proof}

\begin{remark} \label{rmk:off-diag_larger_scale}
  If $T_s$ satisfies $(p_0,q_0)$ off-diagonal estimates at scale $\sqrt{s}$,
  then it satisfies
  \begin{equation}\label{eq:off-diag_larger_scale}
    \left(\fint_{B(r)} \lvert T_s (f\1_{B_1})\rvert^{q_0} \D{\mu} \right)^{1/q_0} \lesssim \rho\Big(\frac{d(B_1,B(r))}{s}\Big) \left( \fint_{B_1} \lvert f\rvert^{p_0} \D{\mu}\right)^{1/p_0},
  \end{equation}
  for balls $B(r)$ of radius $r \ge \sqrt{s}$ and $B_{1}$ of radius $\sqrt{s}$.
  \begin{proof}[Proof of \eqref{eq:off-diag_larger_scale}]
    It's enough to cover the larger ball $B(r)$ with a collection $\mathcal{B}$ of smaller, finite overlapping
    balls of radius $\sqrt{s}$.
    \begin{align*}
      \left( \fint_{B(r)} \lvert T_s(f\mathbbm{1}_{B_{1}}) \rvert^{q_0} \D{\mu} \right)^{1/q_0}
      & = \left( \sum_{B \in \mathcal{B}} \frac{\lvert B\rvert}{\lvert B(r)\rvert } \fint_{B} \lvert T_s(f\mathbbm{1}_{B_{1}}) \rvert^{q_0} \D{\mu} \right)^{1/q_0} \\
      & \le  \left( \sum_{B \in \mathcal{B}} \frac{\lvert B\rvert}{\lvert B(r)\rvert } \right)^{1/q_0}  \left( \sup_{B \in \mathcal{B}} \fint_{B} \lvert T_s(f\mathbbm{1}_{B_{1}}) \rvert^{q_0} \D{\mu} \right)^{1/q_0} \\
      \text{( by \cref{rmk:finite_overlap_balls} )} & \le  \Lambda_{\mathcal{B}}^{1/q_0}  \sup_{B \in \mathcal{B}} \left( \fint_{B} \lvert T_s(f\mathbbm{1}_{B_{1}}) \rvert^{q_0} \D{\mu} \right)^{1/q_0} 
    \end{align*}

    We can use off-diagonal estimates at scale $\sqrt{s}$,
    \begin{equation*}
      \sup_{B \in \mathcal{B}} \left(  \fint_{B} \lvert T_s(f\mathbbm{1}_{B_{1}}) \rvert^{q_0} \D{\mu} \right)^{1/q_0} \lesssim \sup_{B \in \mathcal{B}} \rho\left(\frac{d(B,B_1)}{s}\right)  \left( \fint_{B_1} \lvert f \rvert^{p_0} \D{\mu} \right)^{1/p_0} .
    \end{equation*}      
    and 
    the supremum of $\rho(d(B,B_1)/s)$ over $B \in \mathcal{B}$ is at most
    $\rho\big(d(B(r),B_1)/s\big)$.
  \end{proof}

\end{remark}

We denote the generated averaging operator by $P_t \coloneqq e^{-tL}$.
This is used as an approximation of the identity at scale $\sqrt{t}$,
since for any $p \in (p_0,q_0)$ we have
\begin{equation*}
  \lim_{t \to 0}  \lVert f - e^{-tL}f \rVert_{L^p} = 0 \quad\text{ and } \quad \lim_{ t \to \infty} \lVert e^{-tL} f \rVert_{L^p} = 0 .
\end{equation*}

For $\alpha > 0$, we also consider the family of operators
$Q_t^{(\alpha)} \coloneqq c_{\alpha}^{-1}(tL)^{\alpha}
e^{-tL}$ with $c_{\alpha} = \int^{\infty}_{0} s^{\alpha} e^{-s}\frac{\D{s}}{s}$.
These operators will satisfy an adapted Calderón reproducing formula
for functions in $f \in L^p$ with $p \in (p_0,q_0)$, namely
\begin{equation*}%
  f = \int_0^\infty Q_t^{(\alpha)} f \frac{\D{t}}t .
\end{equation*}
Also define
$$
P_{t}^{(\alpha)} := \int^{\infty}_{1} Q_{st}^{(\alpha)} \, \frac{\D{s}}{s}.
$$
Then $P_t^{(\alpha)}$ is related to the operator
$Q_{t}^{(\alpha)}$ through $t\partial_t P_t^{(\alpha)} = - Q_t^{(\alpha)}$.
We also have that as $L^{p}$-bounded operators,
\begin{equation*}
  P_t^{(\alpha)} = \mathrm{Id} + \int_0^t Q_s^{(\alpha)} \frac{\D{s}}s . 
\end{equation*}

\begin{remark}
  Note that   for any integer $N \in\mathbb{N}$
  the operators $P_t^{(N)}$ and $Q_t^{(N)}$ satisfy off-diagonal estimates at scale $\sqrt{t}$ for all $t>0$.
\end{remark}

\begin{define}[Off-diagonal estimates at all scales]
  \label{def:OffDiagonalAllScales}
  A family of operators $\{T_t\}_{t>0}$ is said to satisfy $(p_0,q_0)$
  off-diagonal estimates at all scales
  if for all balls $B_{1}, \, B_{2}$ of radius $r_1,r_2$ we have
  \begin{equation*}
    \norm{T_{t}}_{L^{p_{0}}(B_{1}) \rightarrow L^{q_{0}}(B_{2})}
    \lesssim \big\lvert B_{1,\sqrt{t}}\big\rvert^{-\frac{1}{p_{0}}}
    \big\lvert B_{2,\sqrt{t}}\big\rvert^{\frac{1}{q_{0}}} \rho \br{\frac{d(B_{1},B_{2})}{\sqrt{t}}},
  \end{equation*}
  where
  $B_{i,\sqrt{t}} := (\sqrt{t}/r_i) B_{i}$ for $i = 1, \, 2$
  and 
  $\rho \colon [0,\infty) \to (0,1] $ is a non increasing function such that
  $\rho(0)=1$ and $\lim_{x\to\infty}\lvert x\rvert^a \rho(x)=0$ for some $a \ge 0$.
\end{define}

It is trivial to see that off-diagonal estimates at all scales implies off-diagonal estimates at scale $\sqrt{t}$. This
stronger condition is used in our cancellation hypothesis, Assumption
\ref{assum:S}(c).

Let $\psi : (0,\infty) \rightarrow (0,\infty)$ be a non-decreasing
function. A space of homogeneous type $(M,\mu)$ is said to be of
$\psi$-growth if
\begin{equation*}
  \abs{B(x,r)} = \mu(B(x,r)) \simeq \psi(r)
\end{equation*}
uniformly for all $x \in M$ and $r > 0$. Notice that this condition is
weaker than \eqref{eqtn:Growth}. For spaces of $\psi$-growth, one encounters
another notion of off-diagonal estimate. These types of estimates are studied in
\cite{AuscherMartellII}. 

\begin{define}[Full off-diagonal estimates]
  \label{def:FullOffDiag}
 Suppose that $(M,\mu)$ is of $\psi$-growth. A family of operators
 $\{T_t\}_{t>0}$ is said to satisfy $(p_0,q_0)$
  full off-diagonal estimates if for all closed sets $E, \, F$ we have
  \begin{equation*}
   \norm{T_{t}}_{L^{p}(E) \rightarrow L^{p}(F)}
    \lesssim \psi(\sqrt{t})^{\frac{1}{q_{0}} - \frac{1}{p_{0}}}\rho
    \br{\frac{d(E,F)}{\sqrt{t}}},
  \end{equation*}
  where $\rho \colon [0,\infty) \to (0,1] $ is a non increasing function such that
  $\rho(0)=1$ and $\lim_{x\to\infty}\lvert x\rvert^a \rho(x)=0$ for some $a \ge 0$.
\end{define}

\begin{rmk}
  \label{rmk:EquivalentOD}
  It is not difficult to show that for spaces of $\psi$-growth,
  the three different notions of off-diagonal estimates, Definitions
  \ref{def:off-diagonal_ests}, \ref{def:OffDiagonalAllScales} and
  \ref{def:FullOffDiag}, are all equivalent for a particular choice of $\rho$.
\end{rmk}

\subsection{Weight classes}
\label{subsec:weight_classes}
We recall some basic properties of the Muckenhoupt and
reverse H\"{o}lder weight classes as defined in the
introduction. Refer to \cite{Johnson1991} for further information.

\begin{lem}
  \label{lem:WeightProperties}
  The following properties of the weight classes $A_{p}$ and $RH_{q}$ are true.
  \begin{enumerate}
  \item[(i)] For $p \in (1,\infty)$, a weight $w$ will be contained in
    the class $A_{p}$ if and only if $w^{1 - p'} \in A_{p'}$. Moreover,
    \begin{equation*}
      \big[w^{1-p'}\big]_{A_{p'}} = \big[w\big]^{p' - 1}_{A_{p}}.
    \end{equation*}
  \item[(ii)] For $q \in [1,\infty]$ and $s \in [1,\infty)$, a weight
    $w$ will be contained in $A_{q} \cap RH_{s}$ if and only if $w^{s}
    \in A_{s(q - 1) + 1}$. Moreover,
    \begin{equation*}
      \max\{\brs{w}^{s}_{A_{q}} , \brs{w}^{s}_{RH_{s}}\} \le \brs{w^{s}}_{A_{s(q - 1) + 1}}    
      \leq \brs{w}^{s}_{A_{q}} \brs{w}^{s}_{RH_{s}}.
    \end{equation*} 

  \end{enumerate}
\end{lem}

For $1\leq p_{0} < 2 < q_{0} \leq \infty$ and  $p \in (p_{0},q_{0})$ define
$$
\phi(p) := \br{\frac{q_{0}}{p}}' \br{\frac{p}{p_{0}} - 1} + 1.
$$
The dependence of $\phi$ on $p_{0}$ and $q_{0}$ will be kept
implicit. From the previous lemma, we get that a weight $w$ will be contained in the class
$A_{\frac{p}{p_{0}}} \cap RH_{(\frac{q_{0}}{p})'}$ if and only if
$w^{(\frac{q_{0}}{p})'}$ is contained in $A_{\phi(p)}$ and it will be
true that
\begin{equation}
  \label{eqtn:WeightProperty}
\brs{w^{(\frac{q_{0}}{p})'}}_{A_{\phi(p)}} \leq
\br{\brs{w}_{A_{\frac{p}{p_{0}}}}  \brs{w}_{RH_{(\frac{q_{0}}{p})'}}}^{(\frac{q_{0}}{p})'}.
\end{equation}

In the article \cite{AuscherMartellIGeneral}, the authors P. Auscher and
J. M. Martell proved a
restricted range extrapolation result that allowed one to obtain
$L^{p}(w)$-boundedness for the full range of $p \in (p_{0},q_{0})$ and
$w \in A_{\frac{p}{p_{0}}} \cap RH_{(\frac{q_{0}}{p})'}$ directly
from the $L^{q}(w)$-boundedness for all $w \in A_{\frac{q}{p_{0}}} \cap RH_{(\frac{q_{0}}{q})'}$ of a single index
$q \in (p_{0},q_{0})$. In their result, they do not state
the dependence of the bound on the weight characteristic
$\brs{w^{(\frac{q_{0}}{p})'}}_{\phi(p)}$. However, through careful
inspection of their proof and by tracing the relevant constants, it is not difficult to see that their
extrapolation result will have the following sharp dependence on this
weight characteristic.

As in \cite{cruz2004extrapolation},
$\mathcal{F}$ denotes a family of ordered pairs of non-negative,
measurable functions $(f, g)$.

\begin{thm}[Sharp Restricted Range Extrapolation {\cite[Thm.~4.9]{AuscherMartellIGeneral}}] 
  \label{thm:Extrapolation} 
  Let $0 < p_{0} < q_{0} \leq \infty$. Suppose that there exists
  $q$ with $p_{0} \leq q < q_{0}$
  such that for $(f,g) \in \mathcal{F}$,  
  \begin{equation*}
    \norm{f}_{L^{q}(w)} \leq C \brs{w^{(\frac{q_{0}}{q})'}}_{A_{\phi(q)}}^{\alpha} \norm{g}_{L^{q}(w)}
    \quad \text{ for all } w\in A_{\frac{q}{p_{0}}} \cap RH_{(\frac{q_{0}}{q})'},
  \end{equation*}
  for some $\alpha > 0$ and $C > 0$ independent of the weight.
  Then, for all $p_{0} < p < q_{0}$ and $(f,g) \in \mathcal{F}$ we have
  
  \begin{equation*}
    \norm{f}_{L^{p}(w)}
    \leq C' \brs{w^{(\frac{q_{0}}{p})'}}_{A_{\phi(p)}}^{\beta(p,q)\cdot\alpha} \norm{g}_{L^{p}(w)}
    \quad \text{ for all } w \in A_{\frac{p}{p_{0}}} \cap RH_{(\frac{q_{0}}{p})'},
  \end{equation*}
  
  where $\beta(p,q) \coloneqq \max \br{1, \frac{(q_{0} - p)(q -
      p_{0})}{(q_{0} - q)(p - p_{0})}}$ and $C' > 0$ is independent
  of the weight.
\end{thm}

\section{Applications}
\label{sec:Applications}

In this section, we consider two distinct applications of our
quadratic sparse domination result and Corollary \ref{cor:Weighted}. For the first application, weighted estimates for square functions associated with
divergence form elliptic operators will be proved. For the particular case of the Laplacian
operator $-\Delta$, this will allow us to recover the estimates from \cite{MR2770437} that are known to be sharp in the $A_{p}$-characteristic constant. The second example that we will look at are square
functions associated with the Laplace--Beltrami operator on a
Riemannian manifold.

\subsection{Elliptic Operators}

Fix $n \in \N \setminus \lb 0 \rb$
and consider the Euclidean space $\mathbb{R}^n$ with the Lebesgue measure.
This is a space of $\psi$-growth, so all definitions of off-diagonal estimates
are equivalent, see \cref{rmk:EquivalentOD}.

Let $A$ be a $n \times n$
matrix-valued function on $\R^{n}$ that is bounded and elliptic in the
sense that
$$
\mathrm{Re} \langle A(x) \xi, \xi
  \rangle_{\C^{n}} \geq \lambda \abs{\xi}^{2},
$$
for some $\lambda > 0$, for all $\xi, \, x \in
\R^{n}$. Consider the divergence
form elliptic operator
$$
L = - \mathrm{div} A \nabla,
$$
defined through its corresponding sesquilinear form as a densely
defined and maximally accretive operator on $L^{2}(\R^{n})$. The
operator $L$ generates an analytic semigroup $\lb e^{-z L} \rb_{z \in
  \Sigma_{\pi/2 - \theta}}$, where
$$
\theta \coloneqq \sup \lb \abs{\mathrm{arg} \langle L f, f \rangle} : f \in \mathcal{D}_{2}(L) \rb.
$$
Let
$g_{L}$ and $G_{L}$ denote the square function operators associated with
$L$ defined by
$$
g_{L}f \coloneqq \br{\int^{\infty}_{0} \abs{(t L)^{\frac{1}{2}} e^{-t L}}^{2} \,
  \frac{\D{t}}{t}}^{\frac{1}{2}} \quad \text{and} \quad G_{L}f \coloneqq \br{\int^{\infty}_{0}
  \abs{\sqrt{t} \nabla e^{-t L}f}^{2} \, \frac{\D{t}}{t}}^{\frac{1}{2}}.
$$
In the
articles \cite{AuscherMartellII} and \cite{AuscherMartellIII},
off-diagonal estimates for operator families associated with these
square functions were
studied in great detail. The below proposition outlines some
properties of such off-diagonal estimates that will be required in
order to apply Corollary \ref{cor:Weighted} to these two square functions.

\begin{prop}[{\cite[Prop.~3.3]{AuscherMartellIII}}]
 \label{prop:OffDiagonalIntervals} 
 For $m \in \N$ and $0 < \mu < \pi/2 - \theta$, there exists maximal intervals
 $\mathcal{J}^{m}(L)$ and $\mathcal{K}^{m}(L)$ in $[1,\infty]$ satisfying the below properties.\\
 
 \begin{enumerate}
 \item[$\bullet$] If $p_{0}, \, q_{0} \in \mathcal{J}^{m}(L)$ with $p_{0}\leq q_{0}$
   then $\lb (z L)^{m} e^{-z L} \rb_{z \in \Sigma_{\mu}}$
   satisfies $(p_{0},q_{0})$ full off-diagonal estimates. \\

 \item[$\bullet$] If $p_{0}, \, q_{0} \in \mathcal{K}^{m}(L)$ with $p_{0}\leq q_{0}$
   then $\lb \sqrt{z} \nabla (z L)^{m} e^{-z L} \rb_{z \in \Sigma_{\mu}}$
   satisfies $(p_{0},q_{0})$ full off-diagonal estimates. \\

 \item[$\bullet$] The interiors $\mathrm{int} \, \mathcal{J}^{m}(L)$
   and $\mathrm{int} \, \mathcal{K}^{m}(L)$ are independent of $m \in
   \N$. \\

 \item[$\bullet$] The inclusion $\mathcal{K}^{m}(L) \subseteq
   \mathcal{J}^{m}(L)$ is satisfied for any $m \in \N$. \\

 \item[$\bullet$] The point $p=2$ is contained in  $\mathcal{K}^{m}(L)$.  
 \end{enumerate}
 \end{prop}

\vspace*{0.1in}

\begin{rmk}
  \label{rmk:HigherOrder}
  Observe that for any $m \geq 1$, $\mathcal{J}^{1}(L) \subset
  \mathcal{J}^{m}(L)$. This follows from the decomposition
  $$
  (t L)^{m} e^{-t L} = (t L) e^{-t L/m} \cdots (t L) e^{-t L / m},
  $$
  the partition
  $$
  [p_{0},q_{0}] = [p_{0},p^{(1)}] \cup [p^{(1)},p^{(2)}] \cup \cdots \cup
  [p^{(m-1)},q_{0}]
  $$
  where $p^{(i)} \coloneqq p_{0} + i (q_{0} - p_{0})/m$, and the property that
  full off-diagonal estimates are stable under composition
  (c.f. {\cite[Thm.~2.3~(b)]{AuscherMartellIII}}).%

  It is also not difficult to see that $\mathcal{J}^{0}(L) \subset
  \mathcal{J}^{1}(L)$. Indeed, consider the expression
  $$
  t L e^{-t L} = e^{-\frac{t}{3}L} \cdot (t L) e^{-\frac{t}{3} L}
  \cdot e^{-\frac{t}{3} L}.
  $$
  For $p_{0}, \, q_{0} \in \mathcal{J}^{0}(L)$ with $p_{0} < 2 < q_{0}$,
  H\"{o}lder's inequality implies that $e^{-\frac{t}{3}L}$ satisfies
  both $(p_{0},2)$ and $(2,q_{0})$ full off-diagonal estimates. It is
  also well-known that $t L e^{-\frac{t}{3} L}$ satisfies $(2,2)$ full
  off-diagonal estimates. The stability of full off-diagonal estimates under
  composition then implies that $t L e^{-t L}$ satisfies
  $(p_{0},q_{0})$ full off-diagonal estimates.
\end{rmk}

 Applying Corollary
\ref{cor:Weighted} to the operators $L$ and $g_{L}$ will produce the
following weighted result.

\begin{prop} 
 \label{prop:Elliptic} 
 Let $p_{0}, \, q_{0} \in \mathcal{J}^{0}(L)$ with $p_{0} < 2 <  q_{0}$. Then, for any $p \in
 (p_{0}, q_{0})$ and $w \in A_{\frac{p}{p_{0}}} \cap RH_{(\frac{q_{0}}{p})'}$,
 $$
\norm{g_{L}}_{L^{p}(w)} \lesssim \br{\brs{w}_{A_{\frac{p}{p_{0}}}}
  \cdot \brs{w}_{RH_{(\frac{q_{0}}{p})'}}}^{\gamma(p)},
$$
where $\gamma(p)$ is as defined in Corollary \ref{cor:Weighted}.
\end{prop}

\begin{proof}  
 To prove the theorem, it is sufficient to check that the hypotheses
 of  Corollary \ref{cor:Weighted}, namely Assumptions \ref{assum:L}
 and \ref{assum:S}, are valid for the operators $L$ and
 $g_{L}$ and the indices $p_{0}, \, q_{0}$. Assumption \ref{assum:L}
 is clearly valid since the definition of $\mathcal{J}^{0}(L)$ implies
 that the semigroup $e^{-t L}$
 will satisfy $(p_{0},q_{0})$ full off-diagonal
 estimates.

 It remains to prove the validity of Assumption \ref{assum:S}.
Part (a), the $L^{2}$-boundedness of $g_{L}$, follows from
 the fact that $L$ possesses a bounded holomorphic functional
 calculus on $L^{2}$. Assumption \ref{assum:S}(b), the off-diagonal
 estimates of the operator family $ (t L)^{\frac{1}{2}} e^{-t L}$, can
 be proved if we combine the fact that $p_{0}, \, q_{0} \in
 \mathcal{J}^{0}(L) \subset \mathcal{J}^{1}(L)$ together with the relation
 
 \begin{equation*}
   (t L)^{\frac{1}{2}} e^{-t L} f = \frac{1}{\sqrt{\pi}} \sqrt{t} \int^{\infty}_{0} L 
   e^{-(s + t)L} f \, \frac{\D{s}}{\sqrt{s}}
 \end{equation*}
and Minkowski's inequality. Assumption \ref{assum:S}(c) follows on
observing that
\begin{align*}\begin{split}  
 \mathcal{Q}_{s} (t L)^{N} e^{-t L} &= (s L)^{\frac{1}{2}} e^{-s L} (t
 L)^{N} e^{-t L} \\
 &= \frac{s^{\frac{1}{2}}t^{N}}{(s + t)^{N + \frac{1}{2}}} ((s + t) L)^{N + \frac{1}{2}} e^{- (s + t)L}
\end{split}\end{align*}
and that since $p_{0}, \, q_{0} \in \mathcal{J}^{1}(L)$ the operator family
$\Theta_{r}^{(N)} = (r L)^{N + \frac{1}{2}} e^{-r L}$ will possess
$(p_{0},q_{0})$ full off-diagonal bounds for any $N \geq N_{0} = 0$ by
the argument of Remark \ref{rmk:HigherOrder}. Finally, for Assumption
\ref{assum:S}(d), in the proof of \cite[Thm.~7.2~(a)]{AuscherMartellIII} %
it was proved that
for any ball $B(x,r)$ we have
\begin{equation}
  \br{\dashint_{B(x,r)} \abs{g_{L} e^{-r^{2}L} f}^{q_{0}} \D{y}}^{\frac{1}{q_{0}}} \lesssim \sum_{j \geq 1} c(j)
  \br{\dashint_{2^{j + 1} B(x,r)} \abs{g_{L}f}^{p_{0}} \D{y}}^{\frac{1}{p_{0}}},\label{eq:AM3_for_g}
\end{equation}
for some sequence of numbers $c(j) > 0$ that satisfies $\sum_{j
  \geq 1} c(j) \lesssim 1$. This clearly implies that
$$
\br{\dashint_{B(x,r)} \abs{g_{L} e^{-r^{2} L}f}^{q_{0}} \D{y}}^{\frac{1}{q_{0}}} \lesssim \inf_{y \in B(x,r)} \mathcal{M}_{p_{0}}(g_{L}f)(y),
$$
and thus Assumption \ref{assum:S}(d) is valid.
\end{proof}

Similarly, Corollary \ref{cor:Weighted} can be applied to the
square function $G_{L}$.

\begin{prop} 
 \label{prop:Elliptic2} 
 Let $p_{0}, \, q_{0} \in \mathcal{K}^{0}(L)$ with $p_{0} < 2 < q_{0}$. Then, for any $p \in
 (p_{0}, q_{0})$ and $w \in A_{\frac{p}{p_{0}}} \cap RH_{(\frac{q_{0}}{p})'}$,
 $$
\norm{G_{L}}_{L^{p}(w)} \lesssim \br{\brs{w}_{A_{\frac{p}{p_{0}}}}
  \cdot \brs{w}_{RH_{(\frac{q_{0}}{p})'}}}^{\gamma(p)}.
$$
\end{prop}

\begin{proof}  
In order to apply Corollary \ref{cor:Weighted}, it is sufficient to
show that $G_{L}$ satisfies Assumptions \ref{assum:S} and
\ref{assum:L}. Assumption \ref{assum:L} is implied by $p_{0}, \, q_{0}
\in \mathcal{K}^{0}(L) \subset \mathcal{J}^{0}(L)$. 

Let us now demonstrate the validity of Assumption \ref{assum:S}. The
$L^{2}$-boundedness of $G_{L}$, Assumption \ref{assum:S}(a), follows from the
ellipticity condition of $A$ and a straightforward integration by
parts argument that can be found in
{\cite[pg.~74]{auscher2007necessary}}. Assumption \ref{assum:S}(b) is
implied by the condition $p_{0}, \, q_{0} \in \mathcal{K}^{0}(L)$. For
Assumption \ref{assum:S}(c), notice that
\begin{align*}\begin{split}  
 \mathcal{Q}_{s} Q_{t}^{(N)} &= \sqrt{s} \nabla e^{-s L} (t L)^{N}
 e^{-t L} \\
 &= \frac{s^{\frac{1}{2}} t^{N}}{(s + t)^{N + \frac{1}{2}}} \sqrt{s +
   t} \nabla \br{(s + t) L}^{N} e^{-(s + t)L} \\
&=:  \frac{s^{\frac{1}{2}} t^{N}}{(s + t)^{N + \frac{1}{2}}} \Theta^{(N)}_{s
+ t}. 
 \end{split}\end{align*}
Also observe that
$$
\Theta^{(N)}_{r} = \sqrt{r} \nabla e^{- r L / 2} (r L)^{N} e^{-r L/2}.
$$
As $p_{0}, \, q_{0} \in \mathcal{K}^{0}(L)$, it follows from
H\"{o}lder's inequality that the
operator family $\sqrt{r} \nabla e^{-r L/2}$ will satisfy $(2,q_{0})$ full off-diagonal estimates. Similarly,
since $\mathcal{K}^{0}(L)  \subset
\mathcal{J}^{N}(L)$ for any $N \geq N_{0} = 0$, the family $(r L)^{N} e^{-r L/2}$ satisfies $(p_{0},2)$ full off-diagonal bounds. It then follows from
the stability of full off-diagonal bounds under composition that the
family $\Theta_{r}^{(N)}$ will satisfy $(p_{0},q_{0})$ full
off-diagonal bounds. This proves that
Assumption \ref{assum:S}(c) is satisfied.

 Finally, for Assumption
 \ref{assum:S}(d), in the proof of \cite[Thm.~7.2~(b)]{AuscherMartellIII} %
 it was proved that for any ball $B(x,r)$ we have
 \begin{equation*}
   \br{\dashint_{B(x,r)} \abs{G_{L} e^{-r^{2}L} f}^{q_{0}} \D{y}}^{\frac{1}{q_{0}}} \lesssim \sum_{j \geq 1} d(j)
   \br{\dashint_{2^{j + 1} B(x,r)} \abs{G_{L}f}^{p_{0}} \D{y}}^{\frac{1}{p_{0}}},
 \end{equation*}
for some sequence of numbers $d(j) > 0$ that satisfies $\sum_{j
  \geq 1} d(j) \lesssim 1$. This clearly implies that
\begin{equation}
  \label{eq:AM3_for_G}
  \br{\dashint_{B(x,r)} \abs{G_{L} e^{-r^{2} L}f}^{q_{0}} \D{y}}^{\frac{1}{q_{0}}} \lesssim \inf_{y \in B(x,r)} \mathcal{M}_{p_{0}}(G_{L}f)(y),
\end{equation}
and thus Assumption \ref{assum:S}(d) is valid.
 \end{proof}

 \begin{rmk}\label{rmk:example_Laplacian}
   For $A = I$ we have $L = -\Delta$ and it is then known that
   $\mathcal{J}^{0}(L) = \mathcal{K}^{0}(L) = [1,\infty]$. We can then take $p_{0} = 1$ and $q_{0} = \infty$ in
   Propositions \ref{prop:Elliptic} and \ref{prop:Elliptic2}. This will produce
   the weighted estimates
   \begin{align*}\begin{split}  
       \norm{g_{-\Delta}}_{L^{p}(w)}, \ \norm{G_{-\Delta}}_{L^{p}(w)}
       &\lesssim \br{\brs{w}_{A_{p}} \brs{w}_{RH_{1}}}^{\max \br{\frac{1}{p
             - 1},\frac{1}{2}}} \\
       &= \brs{w}^{\max \br{\frac{1}{p - 1}, \frac{1}{2}}}_{A_{p}}
     \end{split}\end{align*}
   for all $w \in A_{p} \cap RH_{1} = A_{p}$. For both square functions, it is known that these
   estimates are optimal in the sense that they will not hold
   for an exponent of $\brs{w}_{A_{p}}$ any smaller than the above exponent.  This provides a new proof
   of weighted boundedness of the standard square functions
   associated with $-\Delta$ with optimal dependence on the constant
   $\brs{w}_{A_{p}}$. For the original proof using local mean oscillation
   techniques, the reader is referred to \cite{MR2770437}.
 \end{rmk}

   \begin{rmk}
If $A$ is real-valued then it is known that $\mathcal{J}^{0}(L) =
[1,\infty]$ (c.f. \cite{AuscherMartellIII}). Proposition \ref{prop:Elliptic} will then imply that
$$
 \norm{g_{L}}_{L^{p}(w)} \lesssim \brs{w}^{\max \br{\frac{1}{p - 1}, \frac{1}{2}}}_{A_{p}}
 $$
 for all $w \in A_{p}$. This result was first proved by Bui and Duong
 in \cite{MR4058541}.
     \end{rmk}

     \subsection{Laplace--Beltrami}

Let $M$ be a complete, connected, non-compact Riemannian manifold. It
will be assumed that the Riemannian measure $\mu$ satisfies the volume
doubling property. In addition, it will also be assumed that there exists a function $\psi :(0,\infty)
\rightarrow (0,\infty)$ for which
\begin{equation*}
\abs{B(x,r)} =  \mu(B(x,r)) \simeq \psi(r)
\end{equation*}
uniformly for all $x \in M$ and $r > 0$. That is, the manifold is of
$\psi$-growth. Enforcing this stronger growth condition will allow us
to interchange our different notions of off-diagonal estimates
(c.f. Remark \ref{rmk:EquivalentOD}).
Consider the Laplace--Beltrami operator $\Delta$
defined as an unbounded operator on $L^{2}(M,\mu)$ through the
integration by parts formula
$$
\langle \Delta f, f \rangle = \norm{\abs{\nabla f}}^{2}_{2}
$$
for $f \in C^{\infty}_{0}(M)$, where $\nabla$ is the Riemannian
gradient. The positivity of $\Delta$ implies that it will generate an
analytic semigroup $e^{-t \Delta}$ on
$L^{2}(M,\mu)$.

Recall that the heat kernel $k_{t}(x,y)$ of $\Delta$
is said to satisfy Gaussian upper bounds if there exists $c > 0$ such that
$$
k_{t}(x,y) \lesssim \frac{1}{\abs{B(x,\sqrt{t})}} e^{-c \frac{d^{2}(x,y)}{t}}
$$
for all $x, \, y  \in M$ and $t > 0$. This is a very common
assumption that is imposed when considering the boundedness of singular
operators on Riemannian manifolds. For further information refer to
\cite{coulhon1999riesz}, \cite{auscher2004riesz} or \cite{AuscherMartellIV}. Consider the square function
$g_{\Delta}$ defined through,
\begin{equation*}
  g_{\Delta} f \coloneqq \left(\int^{\infty}_{0} \abs{(t \Delta)^{\frac{1}{2}}e^{-t\Delta}f}^2 \frac{\D{t}}{t}\right)^{1/2}. 
\end{equation*}
The boundedness for square functions of this form on unweighted
$L^{p}(M)$ with $1 < p \leq 2$ was previously proved in
\cite{LPSmanifold}. Let us consider the weighted case on the full
range of $p \in (1,\infty)$.

\begin{prop} 
 \label{prop:LBg} 
 Suppose that the heat kernel for $M$ satisfies Gaussian upper
 bounds. Then, for any $p \in (1,\infty)$ and $w \in A_{p}$,
 $$
\norm{g_{\Delta}}_{L^{p}(w)} \lesssim \brs{w}_{A_{p}}^{\max
  \br{\frac{1}{2}, \frac{1}{p - 1}}}.
 $$
\end{prop}

\begin{proof}  
 This result will follow from \cref{cor:Weighted} provided
 that Assumptions
 \ref{assum:L} and \ref{assum:S} are verified to hold with $p_{0} = 1$ and
 $q_{0} = \infty$.

 For Assumption \ref{assum:L}, it is known that the
 heat kernel satisfying Gaussian upper bounds
 is equivalent to the semigroup
 $e^{-t\Delta}$ satisfying $(1,\infty)$ full off-diagonal
 estimates. For proof, the reader is referred to
 {\cite[Prop.~2.2]{AuscherMartellII}} and
 {\cite[Prop.~3.3]{AuscherMartellII}}. Thus Assumption \ref{assum:L}
 will be valid.

 For Assumption \ref{assum:S}(a), the $L^{2}$-boundedness of $g_{\Delta}$
 follows from the bounded holomorphic
 functional calculus of $\Delta$ on $L^{2}$. For Assumption \ref{assum:S}(b),
 notice that
 $$
(t \Delta)^{\frac{1}{2}} e^{-t \Delta} = e^{-\frac{t}{3} \Delta} \cdot (t
\Delta)^{\frac{1}{2}} e^{-\frac{t}{3} \Delta} \cdot e^{-\frac{t}{3} \Delta}. 
$$
Observe that since the semigroup $e^{-t \Delta}$ satisfies
$(1,\infty)$ full off-diagonal estimates, $e^{-\frac{t}{3} \Delta}$ will
satisfy both $(1,2)$ and
$(2,\infty)$ full off-diagonal bounds by H\"{o}lder's inequality. At
the same time, $(t\Delta)^{\frac{1}{2}}e^{-t\Delta}$
is well-known to satisfy $(2,2)$ full off-diagonal
bounds (c.f. {\cite[pg.~930]{auscher2004riesz}} and the references
therein). It then follows from the stability of full off-diagonal
bounds under composition ({\cite[Thm.~2.3~(b)]{AuscherMartellII}}) %
that $(t \Delta)^{\frac{1}{2}} e^{-t\Delta}$ satisfies $(1,\infty)$ full off-diagonal bounds.
This proves that Assumption \ref{assum:S}(b) is satisfied.

Assumption \ref{assum:S}(c) follows from the expression
$$
\mathcal{Q}_{s} (t\Delta)^{N} e^{-t \Delta} = \frac{s^{\frac{1}{2}}
  t^{N}}{(s + t)^{N + \frac{1}{2}}} \brs{(s + t) \Delta}^{N +
  \frac{1}{2}} e^{-(s + t)\Delta}
$$
and the fact that
the operator family $\{(r \Delta)^{N +  \frac{1}{2}}e^{- r \Delta}\}_{r>0}$
satisfies $(1,\infty)$ full off-diagonal
bounds by an argument similar to that of Remark
\ref{rmk:HigherOrder}.

Finally, the validity of Assumption
\ref{assum:S}(d)  can be proved in an identical manner to the argument
used to obtain \eqref{eq:AM3_for_g}.
This argument can be found in {\cite[\S
  7]{AuscherMartellIII}} on pages 729--730. This argument in the elliptic
setting follows from a combination of the off-diagonal estimates of
the constituent operators, the fact that the constituent operators are
expressible in terms of the semigroup and a variation of the
Marcinkiewicz--Zygmund theorem \cite{GrafakosClassical}*{Thm.~5.5.1}. All three of these components will
hold for our square function in this Riemannian manifold setting and thus the argument will be valid.
 \end{proof}

 Next, we will apply our sparse result to the square function
 \begin{equation*}
   G_{\Delta}f \coloneqq \left(\int^{\infty}_{0} \abs{\sqrt{t} \nabla e^{-t\Delta}}^{2} \frac{\D{t}}{t}\right)^{1/2}.
 \end{equation*}
Define
$$
q_{+} \coloneqq \sup \lb p \in (1,\infty) : \norm{ \abs{ \nabla
  \Delta^{-\frac{1}{2}} f }}_{p} \lesssim \norm{f}_{p}  \rb.
$$
The weighted boundedness of the Riesz transforms operator $\nabla
\Delta^{-\frac{1}{2}}$ on $L^{p}(M,w\D{\mu})$ was considered for $p \in
(1,q_{+})$ in \cite{AuscherMartellIV}. Owing to the strong connection
between the Riesz transforms and the square function $G_{\Delta}$,
the range $(1,q_{+})$ will also be a natural
interval over which to consider the boundedness of $G_{\Delta}$.
It was proved in \cite{coulhon1999riesz} that $q_{+} \geq 2$. In the below proposition we
assume this inequality to be strict.

\begin{prop} 
  \label{prop:LBG}
  Assume that the heat kernel of $M$ satisfies Gaussian upper bounds and that $q_{+} >
  2$. Let $2 < q_{0} < q_{+}$ and $p \in [1,q_{0})$. Then for any $w \in A_{p} \cap RH_{(\frac{q_{0}}{p})'}$,
 $$
\norm{G_{\Delta}}_{L^{p}(w)} \lesssim
\br{\brs{w}_{A_{p}} \cdot \brs{w}_{RH_{(\frac{q_{0}}{p})'}}}^{\gamma(p)}.
 $$
\end{prop}

\begin{proof}  
 Once again, let us apply Corollary \ref{cor:Weighted}. Assumption
 \ref{assum:L} will be true for the same reason as in Proposition
 \ref{prop:LBg}. As $q_{+} > 2$, Assumption \ref{assum:S}(a) is implied by the
 boundedness of $\nabla \Delta^{-\frac{1}{2}}$ on $L^{2}$ and the
 boundedness of $g_{\Delta}$ on $L^{2}$.

 Let us show that the family of operators $\mathcal{Q}_{t} = \sqrt{t}
 \nabla e^{-t \Delta}$ satisfies $(1,q_{0})$ off-diagonal
 estimates at all scales with $\rho(x) = \exp(-cx^2)$, for some $c>0$.
 Since we have imposed the uniform $\psi$-growth condition
 upon our manifold, this is equivalent to proving $(1,q_{0})$
 off-diagonal estimates at scale $\sqrt{t}$. Fix balls $B_{1}, \, B_{2} \subset M$ of radius $\sqrt{t}$.
 From the argument in the proof of {\cite[Prop~1.10]{auscher2004riesz}}, %
 $$
\br{\int_{M} \abs{\nabla_{x} k_{t}(x,y)}^{q_{0}} e^{c
    \frac{d^{2}(x,y)}{t}} \D{\mu}(x)}^{\frac{1}{q_{0}}} \lesssim
\frac{1}{\sqrt{t} \abs{B(y,\sqrt{t})}^{1 - \frac{1}{q_{0}}}}
$$
for all $t > 0$ and $y \in M$, where $c > 0$ is dependent on
$q_{0}$. This immediately implies that
\begin{align*}\begin{split}  
 \br{\dashint_{B_{2}} \abs{\nabla_{x} k_{t}(x,y)}^{q_{0}} \D{\mu}(x)}^{\frac{1}{q_{0}}} &\lesssim \frac{1}{\sqrt{t}} e^{-c
  \frac{d^{2}(B_{1},B_{2})}{t}} \frac{1}{\abs{B(y,\sqrt{t})}^{1 -
    \frac{1}{q_{0}}} \abs{B_{2}}^{\frac{1}{q_{0}}}} \\
&\simeq \frac{1}{\sqrt{t}} e^{-c  \frac{d^{2}(B_{1},B_{2})}{t}} \frac{1}{\psi(\sqrt{t})},
\end{split}\end{align*}
where the last line follows from the uniform $\psi$-growth condition
imposed upon our manifold. For $f$ supported in $B_{1}$, Minkowski's
inequality followed by the previous estimate produces
\begin{align*}\begin{split}  
 \br{\dashint_{B_{2}} \abs{\sqrt{t} \nabla e^{-t \Delta}f(x)}^{q_{0}}
   \D{\mu}(x)}^{\frac{1}{q_{0}}} &= \br{\dashint_{B_{2}}
   \abs{\int_{B_{1}} \sqrt{t} \nabla_{x} k_{t}(x,y) f(y) \D{\mu}(y)}^{q_{0}} \D{\mu}(x)}^{\frac{1}{q_{0}}} \\
 &\leq \int_{B_{1}} \br{\dashint_{B_{2}} \abs{\sqrt{t} \nabla_{x}
     k_{t}(x,y)}^{q_{0}} \D{\mu}(x)}^{\frac{1}{q_{0}}} \abs{f(y)} \D{\mu}(y) \\
 &\lesssim \frac{1}{\psi(\sqrt{t})} e^{- c\frac{d^{2}(B_{1},B_{2})}{t}} \int_{B_{1}} \abs{f(y)} \D{\mu}(y)
 \\
 &\simeq e^{-c\frac{d^{2}(B_{1},B_{2})}{t}} \dashint_{B_{1}} \abs{f(y)} \D{\mu}(y). 
 \end{split}\end{align*}

Let's now prove that Assumption \ref{assum:S}(c) is valid. Observe that
$$
\mathcal{Q}_{s} (t \Delta)^{N} e^{- t \Delta} = \frac{s^{\frac{1}{2}}
  t^{N}}{(s + t)^{N + \frac{1}{2}}} \sqrt{s + t} \nabla e^{-\frac{s +
    t}{2} \Delta} \brs{(s + t) \Delta}^{N} e^{-\frac{s + t}{2} \Delta}
=: \frac{s^{\frac{1}{2}} t^{N}}{(s + t)^{N + \frac{1}{2}}} \Theta^{(N)}_{s +
t}.
$$
Recall from the proof of Proposition \ref{prop:LBg} that the operator
family $\{(r \Delta)^{N} e^{-r \Delta}\}_{r>0}$ satisfies $(1,\infty)$ full
off-diagonal estimates. By H\"{o}lder's inequality it will then also satisfy
$(1,2)$ full off-diagonal estimates. Similarly by H\"{o}lder, the
family $\{\sqrt{r}\nabla e^{-r \Delta}\}_{r>0}$ satisfies $(2,q_{0})$ full
off-diagonal estimates. The
stability of full off-diagonal estimates under composition then
implies that the
operator family $\Theta_{r}$ satisfies $(1,q_{0})$ full off-diagonal
estimates. This proves Assumption \ref{assum:S}(c).

Finally, the validity of Assumption \ref{assum:S} (d)
can be proved in an identical manner to the argument
used to obtain \eqref{eq:AM3_for_G}.
This argument can be found in {\cite[\S 7]{AuscherMartellIII}} on page
732. This argument in the elliptic
setting follows from a combination of the off-diagonal estimates of
the constituent operators, the fact that the constituent operators are
expressible in terms of the semigroup and a variation of the
Marcinkiewicz--Zygmund theorem \cite{GrafakosClassical}*{Thm.~5.5.1}. All three of these components will
hold for our square function in this Riemannian manifold setting and thus the argument will be valid.
 \end{proof}

\section{Boundedness of the Maximal Function}
\label{sec:Maximal}
Throughout this section, fix $p_{0}, \, q_{0} \in [1,\infty]$,
$N_{0} \in \N$ and operators $L$ and $S$ satisfying Assumptions \ref{assum:L} and \ref{assum:S} for such a choice of  $p_{0}, \, q_{0} $.
For a ball $B$ we denote by $r(B)$ its radius.
Define the following maximal operator associated with the square
function,
\begin{align*}
  S^{*}f(x) &\coloneqq \sup_{\substack{B \text{ ball} \\ B \ni x}}
  \br{\dashint_{B} \abs{S^{[r(B)^{2},\infty)}f}^{q_{0}} \D{\mu}}^{1/q_{0}} \\
           &:= \sup_{\substack{B \text{ ball} \\ B \ni x}} \br{\dashint_{B}
  \br{\int^{\infty}_{r(B)^{2}} \abs{\mathcal{Q}_{t} f}^{2}
  \frac{\D t}{t}}^{\frac{q_{0}}{2}} \D{\mu}}^{1 / q_{0}}.
\end{align*}
In this section, our aim is to prove the following boundedness result
for $S^{*}$. 

\begin{thm} 
 \label{thm:Maximal} 
  The maximal function $S^{*}$ is bounded on
 $L^{2}$ and weak-type $(p_{0},p_{0})$.
\end{thm}

 The boundedness of this maximal function constitutes an important part
 of our sparse domination argument. The reliance of our argument on an associated maximal function is a well-known
 method for obtaining sparse bounds and finds its origins in the work of Lacey \cite{MR3625108}.
 It was later streamlined by Lerner \cite{MR3484688}.
 Quite often, the issue of proving sparse domination
 for a particular operator can be reduced to determining an
 appropriate associated maximal operator, proving its (weak) boundedness and then
 applying a stopping time argument that utilises this boundedness.

\subsection{A Pointwise Estimate}
\label{subsec:pointwise_estimate}
In order to prove the boundedness of the operator $S^{*}$ we will
require a couple of preliminary lemmas.
Given a ball $B$,
we define the average of a function $f$ over the annulus $S_k(B) \coloneqq 2^{k+1} B \setminus 2^k B$ for $k\in\mathbb{N}$
as the integral over $S_k(B)$ normalised by $\lvert 2^k B\rvert$.

Recall that $A_0$ is a positive number defined in \cref{assum:S} (c).
\begin{lem}
  \label{lem:OffDiagonal}
  For any $0 < s < r^{2} < t$ and $N \in \N$,
  \begin{equation*}
    \br{\dashint_{B} \abs{\mathcal{Q}_{t} Q_{s}^{(N)}f}^{q_{0}}\D{\mu}}^{\frac{1}{q_{0}}}
    \lesssim \frac{t^{A_{0}} s^{N}}{(s +
      t)^{A_{0} + N}}\br{\frac{\sqrt{t}}{r}}^{\frac{\nu}{q_{0}}}
    \sum_{j \geq 0} 2^{-j} \br{\dashint_{S_{j}(\widetilde{B})} \abs{f}^{p_{0}} \D{\mu}}^{\frac{1}{p_{0}}}
  \end{equation*}
  for any ball $B$ of radius $r$ and $\widetilde{B} := \frac{\sqrt{t}}{r}B$.
\end{lem}

\begin{proof}
Fix $B$ a ball of radius $r$. For $j \geq 0$, let $\mathcal{R}_{j}$ denote a collection of finite
  overlapping balls of radius $\sqrt{t}$ that is a cover for the set
  $S_{j}(\widetilde{B})$. Then,
  \cref{assum:S} (c) together with the triangle inequality
  produces
  \begin{align}\begin{split}
      \label{eqtn:lem:OffDiagonal1}
      &\br{\dashint_{B} \abs{\mathcal{Q}_{t} Q_{s}^{(N)}f}^{q_{0}}\D{\mu}}^{\frac{1}{q_{0}}} = \frac{t^{A_{0}} s^{N}}{(s + t)^{A_{0} +
          N}} \br{\dashint_{B} \abs{\Theta_{s + t}^{(N)}f}^{q_{0}}\D{\mu}}^{\frac{1}{q_{0}}} \\
      & \qquad \qquad \leq \frac{t^{A_{0}} s^{N}}{(s + t)^{A_{0} + N}}
      \sum_{j \geq 0} \sum_{R \in \mathcal{R}_{j}}
      \br{\dashint_{B} \abs{\Theta_{s + t}^{(N)}
          \br{\mathbbm{1}_{R}f}}^{q_{0}} \D{\mu}}^{\frac{1}{q_{0}}} \\
      & \qquad \qquad \lesssim \frac{t^{A_{0}} s^{N}}{(s + t)^{A_{0} + N}} \sum_{j \geq 0}
      \sum_{R \in \mathcal{R}_{j}}
      \frac{\abs{B}^{-\frac{1}{q_{0}}}
        \abs{R}^{\frac{1}{p_{0}}}}{\abs{B_{\sqrt{s + t}}}^{-\frac{1}{q_{0}}} \abs{R_{\sqrt{s + t}}}^{\frac{1}{p_{0}}}}
      \br{1 + \frac{d(B,R)^{2}}{s + t}}^{- \frac{\nu + 1}{2}}
      \br{\dashint_{R} \abs{f}^{p_{0}}\D{\mu}}^{\frac{1}{p_{0}}}.
    \end{split}\end{align}
  On utilising the doubling property of our metric space and subsequently $s + t \simeq t$,
\begin{align}\begin{split}
    \label{eqtn:lem:OffDiagonal2}
    \abs{B_{\sqrt{s + t}}}     &\lesssim \br{\frac{\sqrt{s + t}}{r}}^{\nu} \abs{B} \\
    &\simeq \br{\frac{\sqrt{t}}{r}}^{\nu} \abs{B}.
  \end{split}\end{align}
This, together with the fact that $\abs{R} \leq \abs{R_{\sqrt{s + t}}}$
gives
\begin{align}\begin{split}  
    \label{eqtn:OffDiagonal21}
    &\br{\dashint_{B} \abs{\mathcal{Q}_{t} Q_{s}^{(N)}f}^{q_{0}}\D{\mu}}^{\frac{1}{q_{0}}} \\ 
 & \qquad \qquad \lesssim \frac{t^{A_{0}} s^{N}}{(s +
   t)^{A_{0} + N}} \br{\frac{\sqrt{t}}{r}}^{\frac{\nu}{q_{0}}} \sum_{j
   \geq 0} \sum_{R \in \mathcal{R}_{j}}
 \br{1 + \frac{d(B,R)^2}{s + t}}^{-\frac{\nu + 1}{2}} \br{\dashint_{R}
   \abs{f}^{p_{0}} \D{\mu}}^{\frac{1}{p_{0}}}.
\end{split}\end{align}
For $R \in \mathcal{R}_{j}$, since $d(B,R) \geq (2^{j} - 1)
\sqrt{t} \simeq (2^{j} - 1) \sqrt{s + t}$ for $j \geq 1$, we have
\begin{equation}
  \label{eqtn:OffDiagonal211}
\br{1 + \frac{d(B,R)^{2}}{s + t}}^{-\frac{\nu + 1}{2}}
\lesssim 2^{-j(\nu + 1)}.
\end{equation}

Let $\Omega = S_j(\widetilde{B})$ and $\mathcal{R}_{j}$ as defined above in this proof.
The inclusion $ \Omega \subset 2^{j + 1} \widetilde{B} \subset 2^{j + 2}R$
holds for any $R \in \mathcal{R}_{j}$ and $j \in \N$. Thus \cref{lemma:counting_R_balls} implies that

$$
\sum_{R \in \mathcal{R}_{j}} \br{\dashint_{R}
   \abs{f}^{p_{0}} \D{\mu}}^{\frac{1}{p_{0}}} \lesssim 2^{j \nu} \br{\dashint_{S_{j}(\widetilde{B})} \abs{f}^{p_{0}} \D{\mu}}^{\frac{1}{p_{0}}}.
$$
Applying this estimate and \eqref{eqtn:OffDiagonal211} to
\eqref{eqtn:OffDiagonal21} gives us our result.
\end{proof}

Using the previous lemma, the following result can then be proved
using an argument identical to the first estimate of
{\cite[Lem.~4.1]{bernicot2016sharp}}. 
 
\begin{lem}
  \label{lem:BFPOffDiag}
  Fix $N \in \N$ with $N > \mathrm{max}(3 \nu/2 + 1, N_{0})$. For
  any ball $B$ of radius $r(B) > 0$ and $t > r(B)^{2}$ we have
  \begin{equation}
    \label{eqtn:BFPOffDiag1}
\br{\dashint_{B} \abs{\mathcal{Q}_{t} (I -    P_{r(B)^{2}}^{(N)})f}^{q_{0}} \D{\mu}}^{\frac{1}{q_{0}}} \lesssim
\br{\frac{r(B)^{2}}{t}}^{\frac{N}{2}} \sum_{l \geq 0} 2^{-l}
\br{\dashint_{2^{l} B} \abs{f}^{p_{0}} \D{\mu}}^{\frac{1}{p_{0}}}.
\end{equation}
 \end{lem}

 Let $S^{\#}$ denote the maximal operator
 \begin{equation*}
   S^{\#}f(x) := \sup_{\substack{B \text{ ball} \\ B \ni x}}
   \br{\dashint_{B} \abs{S P^{(N)}_{r(B)^{2}}}^{q_{0}} \D{\mu}}^{\frac{1}{q_{0}}}.
 \end{equation*}
This operator was introduced in \cite{bernicot2016sharp} and
formed an important part of their sparse domination argument. 

 \begin{prop}
   \label{prop:Pointwise}
   For every $x \in M$,
\begin{equation*}\label{eq:pointwise_estimate_maximal_truncations}
  S^{*}f(x) \lesssim S^{\#}f(x) + \mathcal{M}_{p_{0}}f(x).
\end{equation*}   
   \end{prop}

\begin{proof}  
  For $x \in
M$ and ball
$B \subset M$ containing $x$, the triangle inequality implies
  \begin{align*}\begin{split}  
 \br{\dashint_{B} \br{\int^{\infty}_{r(B)^{2}}
     \abs{\mathcal{Q}_{t} f}^{2} \frac{\D{t}}{t}}^{\frac{q_{0}}{2}}
   \D{\mu}}^{\frac{1}{q_{0}}} &\leq \br{\dashint_{B}
   \br{\int^{\infty}_{r(B)^{2}} \abs{\mathcal{Q}_{t} (I -
       P_{r(B)^{2}}^{(N)})f}^{2} \, \frac{\D{t}}{t}}^{\frac{q_{0}}{2}} \D{\mu}}^{\frac{1}{q_{0}}} \\
 & \qquad + \br{\dashint_{B}
   \br{\int^{\infty}_{r(B)^{2}}\abs{\mathcal{Q}_{t}
       P_{r(B)^{2}}^{(N)}f}^{2} \, \frac{\D{t}}{t}}^{\frac{q_{0}}{2}} \D{\mu}}^{\frac{1}{q_{0}}}.
\end{split}\end{align*}
For the first term, apply Minkowski's inequality followed by Lemma
\ref{lem:BFPOffDiag} to obtain
\begin{align*}\begin{split}  
 &\br{\dashint_{B} \br{\int^{\infty}_{r(B)^{2}}
     \abs{\mathcal{Q}_{t} (I - P_{r(B)^{2}}^{(N)})f}^{2} \,
     \frac{\D{t}}{t}}^{\frac{q_{0}}{2}} \D{\mu}}^{\frac{1}{q_{0}}} \\ 
 & \qquad \qquad \le \br{\int^{\infty}_{r(B)^{2}} \br{\dashint_{B}
     \abs{\mathcal{Q}_{t} (I - P_{r(B)^{2}}^{(N)})f}^{q_{0}} \D{\mu}}^{\frac{2}{q_{0}}} \, \frac{\D{t}}{t}}^{\frac{1}{2}} \\
 & \qquad \qquad \lesssim \br{\int^{\infty}_{r(B)^{2}}
   \br{\br{\frac{r(B)^{2}}{t}}^{\frac{N}{2}} \sum_{l \geq 0} 2^{-l} \br{\dashint_{2^{l} B} \abs{f}^{p_{0}} \D{\mu}}^{\frac{1}{p_{0}}}}^{2} \, \frac{\D{t}}{t}}^{\frac{1}{2}} \\
 & \qquad \qquad = r(B)^{N} \br{\int^{\infty}_{r(B)^{2}}
   \frac{\D{t}}{t^{N + 1}}}^{\frac{1}{2}} \sum_{l \geq 0} 2^{-l} \br{\dashint_{2^{l} B} \abs{f}^{p_{0}} \D{\mu}}^{\frac{1}{p_{0}}} \\
 & \qquad \qquad \lesssim \mathcal{M}_{p_{0}}f(x).
\end{split}\end{align*}
For the second term,
\begin{align*}\begin{split}  
 \br{\dashint_{B}
   \br{\int^{\infty}_{r(B)^{2}} \abs{\mathcal{Q}_{t} P_{r(B)^{2}}^{(N)}f}^{2} \, \frac{\D{t}}{t}}^{\frac{q_{0}}{2}} \D{\mu}
 }^{\frac{1}{q_{0}}} &\leq \br{\dashint_{B} \br{\int^{\infty}_{0}
   \abs{\mathcal{Q}_{t} P^{(N)}_{r(B)^{2}}f}^{2} \,
   \frac{\D{t}}{t}}^{\frac{q_{0}}{2}} \D{\mu}}^{\frac{1}{q_{0}}} \\
&= \br{\dashint_{B} \abs{S P^{(N)}_{r(B)^{2}} f}^{q_{0}} \D{\mu}}^{\frac{1}{q_{0}}} \\
&\leq S^{\#}f(x).
\end{split}\end{align*}
We thus obtain the pointwise estimate
\eqref{eq:pointwise_estimate_maximal_truncations}.
 \end{proof}

\subsection{Cancellation of \texorpdfstring{$S$}{S} with respect to \texorpdfstring{$L$}{L}} 

As the operator $\mathcal{M}_{p_{0}}$ is $L^{2}$-bounded and weak-type
$(p_{0},p_{0})$, the pointwise bound of the previous section implies
that in order to prove Theorem \ref{thm:Maximal} it will be sufficient to show that $S^{\#}$ is $L^{2}$-bounded
and weak-type $(p_{0},p_{0})$. According to
{\cite[Prop.~4.6]{bernicot2016sharp}}, $S^{\#}$ will be
$L^{2}$-bounded and weak-type $(p_{0},p_{0})$
if $S$ satisfies the assumptions of \cite{bernicot2016sharp}. The only
assumption from \cite{bernicot2016sharp} that is not included in
our hypotheses is Assumption 1.1(b) of \cite{bernicot2016sharp}, the cancellative property of $S$
with respect to $L$. Instead, for us, the cancellation has
been imposed upon the constituent operators $\mathcal{Q}_{t}$. In this
section it will be proved that cancellation on $\mathcal{Q}_{t}$ with
respect to $L$ implies cancellation on $S$ with respect to $L$.

\begin{prop} 
 \label{prop:OffDiagonal} 
  There exists $\widetilde{N}_{0} \geq N_{0}$ such that for all integers $N
 \geq \widetilde{N}_{0}$, $t > 0$
 and balls $B_{1}, \, B_{2}$ of radius $\sqrt{t}$,
 \begin{equation}
   \label{eqtn:OffDiagonalMain}
\br{\dashint_{B_{2}} \abs{S Q_{t}^{(N)}f}^{q_{0}} \D{\mu}}^{\frac{1}{q_{0}}}
\lesssim \br{1 + \frac{d(B_{1},B_{2})^{2}}{t}}^{- \frac{\nu + 1}{2}}
\br{\dashint_{B_{1}} \abs{f}^{p_{0}} \D{\mu}}^{\frac{1}{p_{0}}}
\end{equation}
for all $f \in L^{p_{0}}(B_{1})$.
\end{prop}

\begin{proof}  
  For $I \subset [0,\infty)$, define the operator  
  \begin{equation*}
    S^{I} f(x) := \br{\int_{I} \abs{\mathcal{Q}_{s} f}^{2} \frac{\D{s}}{s}}^{\frac{1}{2}}.
  \end{equation*}
  
  In order to prove \eqref{eqtn:OffDiagonalMain}, it is sufficient to
  show that a similar estimate holds for the operators $S^{[0,t]}$ and
  $S^{[t,\infty)}$.

  For $I \subset [0,\infty)$, Minkowski's
inequality implies that
\begin{align*}\begin{split}  
    \br{\dashint_{B_{2}} \abs{S^{I} Q_{t}^{(N)}f}^{q_{0}} \D{\mu}}^{\frac{1}{q_{0}}} &= \br{\dashint_{B_{2}} \br{\int_{I}
        \abs{\mathcal{Q}_{s} Q_{t}^{(N)}f}^{2} \frac{\D s}{s}}^{\frac{q_{0}}{2}} \D \mu}^{\frac{2}{q_{0}}\frac12} \\
    &\le \brs{\int_{I} \br{\dashint_{B_{2}} \abs{\mathcal{Q}_{s}
          Q_{t}^{(N)}f}^{q_{0}} \D \mu}^{\frac{2}{q_{0}}} \frac{\D s}{s}}^{\frac{1}{2}}.
  \end{split}\end{align*}
From Assumption \ref{assum:S}(c) and the growth property \eqref{eqtn:Growth}, we have
\begin{align*}\begin{split}
 &\br{\dashint_{B_{2}} \abs{S^{I} Q_{t}^{(N)}f}^{q_{0}} \D{\mu}}^{\frac{1}{q_{0}}} \leq \brs{\int_{I} \frac{s^{2 A_{0}} t^{2
      N}}{(s + t)^{2(A_{0} + N)}} \br{\dashint_{B_{2}} \abs{\Theta_{s
        + t}^{(N)}f}^{q_{0}} \D{\mu}}^{\frac{2}{q_{0}}} \,
  \frac{\D{s}}{s}}^{\frac{1}{2}} \\
&\lesssim \brs{\int_{I} \frac{s^{2 A_{0}} t^{2 N}}{(s + t)^{2
      (A_{0} + N)}} \frac{\abs{B_{1}}^{\frac{2}{p_{0}}} \cdot
    \abs{B_{2}}^{-\frac{2}{q_{0}}}}{\abs{B_{1,\sqrt{s +t}}}^{\frac{2}{p_{0}}} \abs{B_{2,\sqrt{s + t}}}^{-\frac{2}{q_{0}}}} \br{1 + \frac{d(B_{1},B_{2})^{2}}{s +
      t}}^{-(\nu + 1)} \, \frac{\D{s}}{s}}^{\frac{1}{2}} \br{\dashint_{B_{1}} \abs{f}^{p_{0}}
  \D{\mu}}^{\frac{1}{p_{0}}} \\
&\simeq \brs{\int_{I} \frac{s^{2 A_{0}} t^{2 N}}{(s + t)^{2
      (A_{0} + N)}} \varphi\br{\frac{\sqrt{t}}{\sqrt{s + t}}}^{2 \br{\frac{1}{p_{0}} - \frac{1}{q_{0}}}} \br{1 + \frac{d(B_{1},B_{2})^{2}}{s +
      t}}^{-(\nu + 1)} \, \frac{\D{s}}{s}}^{\frac{1}{2}} \br{\dashint_{B_{1}} \abs{f}^{p_{0}}
  \D{\mu}}^{\frac{1}{p_{0}}}.
\end{split}\end{align*}
The property that $\varphi(a) \leq 1$ for $a \leq 1$ then gives
\begin{equation}
   \label{eqtn:OffDiagonal1}
   \br{\dashint_{B_{2}} \abs{S^{I} Q_{t}^{(N)}f}^{q_{0}} \D{\mu}}^{\frac{1}{q_{0}}}
   \lesssim  \brs{\int_{I} \frac{s^{2 A_{0}} t^{2 N}}{(s + t)^{2
      (A_{0} + N)}} \br{1 + \frac{d(B_{1},B_{2})^{2}}{s +
      t}}^{-(\nu + 1)} \, \frac{\D{s}}{s}}^{\frac{1}{2}} \br{\dashint_{B_{1}} \abs{f}^{p_{0}}
  \D{\mu}}^{\frac{1}{p_{0}}}.
  \end{equation}
In order to prove the desired off-diagonal estimate, it is then sufficient to prove
\begin{align}\begin{split}  
    \label{eqtn:AI}
    A_{I} &:=    \int_{I} \frac{s^{2 A_{0}}t^{2N}}{(s + t)^{2
      (A_{0} + N)}} \br{1 + \frac{d(B_{1},B_{2})^{2}}{s +
      t}}^{-(\nu + 1)} \, \frac{\D{s}}{s} \\ &\lesssim  
\br{1 + \frac{d(B_{1},B_{2})^{2}}{t}}^{- (\nu + 1)},
 \end{split}\end{align}
for both intervals $I = [0,t]$ and $I = [t,\infty)$.
Consider first the interval $I = [0,t]$. For $s$ contained in $[0,t]$ we will have $s + t \leq 2 t$ and therefore  
$$
\br{1 + \frac{d(B_{1},B_{2})^{2}}{s + t}}^{- (\nu + 1)} \lesssim \br{1
+ \frac{d(B_{1},B_{2})^{2}}{t}}^{- (\nu + 1)}.
$$
This gives
\begin{align*}\begin{split}
    A_{I} &\lesssim \br{1 + \frac{d(B_{1},B_{2})^{2}}{t}}^{- (\nu +
      1)} 
    \int_{0}^{t} \frac{s^{2 A_{0}} t^{2N}}{(s + t)^{2(A_{0} + N)}} \,
    \frac{\D{s}}{s} \\
    &\leq \br{1 + \frac{d(B_{1},B_{2})^{2}}{t}}^{-(\nu + 1)}
     \frac{1}{t}
    \int^{t}_{0} \D{s} \\
    &= \br{1 + \frac{d(B_{1},B_{2})^{2}}{t}}^{-(\nu + 1)}.
\end{split}\end{align*}
Applying this to \eqref{eqtn:OffDiagonal1} produces the desired
off-diagonal bounds for the operator $S^{[0,t]}$.

Next, let's prove off-diagonal bounds for the operator
$S^{[t,\infty)}$. Suppose first that $t >
d(B_{1},B_{2})^{2}$. When this occurs, note that
\begin{equation}
  \label{eqtn:OffDiagonal3}
\br{1 + \frac{d(B_{1},B_{2})^{2}}{t}}^{-(\nu + 1)} \simeq 1.
\end{equation}
We then have,
\begin{align*}\begin{split}  
 A_{I} &\leq 
\int^{\infty}_{t} \frac{s^{2 A_{0}} t^{2 N}}{(s + t)^{2(A_{0} + N)}}
\, \frac{\D{s}}{s} \\
&\leq t^{2 N}
\int^{\infty}_{t} \frac{1}{(s + t)^{2 N + 1}} \D{s} \\
&\simeq 1 \\
&\simeq  \br{1 +
  \frac{d(B_{1},B_{2})^{2}}{t}}^{- (\nu + 1)}.
 \end{split}\end{align*}
Applying this to \eqref{eqtn:OffDiagonal1} produces the desired
off-diagonal estimates for $S^{[t,\infty)}$.

Finally, we must prove off-diagonal decay for $S^{[t,\infty)}$ for
the case $t \leq d(B_{1},B_{2})^{2}$. We have,
\begin{align*}\begin{split}  
 A_{I} &=
 \int^{\infty}_{t} \frac{s^{2 A_{0}} t^{2 N}}{(s + t)^{2 (A_{0} + N)}}
 \br{1 + \frac{d(B_{1},B_{2})^{2}}{s + t}}^{-(\nu + 1)} \,
 \frac{\D{s}}{s} \\
 &\leq \frac{t^{2
     N}}{d(B_{1},B_{2})^{2 (\nu + 1)}} \int^{\infty}_{t} \frac{\D{s}}{(s
   + t)^{2 N + 1 - (\nu + 1)}}.
\end{split}\end{align*}
Select $\widetilde{N}_{0} \geq N_{0}$ large enough so that $N \geq \widetilde{N}_{0}$ implies $2N >
\nu + 1$. Then,
\begin{align*}\begin{split}  
 A_{I} &\lesssim \frac{t^{\nu + 1}}{d(B_{1},B_{2})^{2 (\nu + 1)}} \\
 &\lesssim  \br{1
 + \frac{d(B_{1},B_{2})^{2}}{t}}^{-(\nu + 1)},
 \end{split}\end{align*}
where the last line follows from the condition $t \leq
d(B_{1},B_{2})^{2}$. Applying this to \eqref{eqtn:OffDiagonal1}
completes our proof.
\end{proof}

The below corollary, in combination with the pointwise estimate
\cref{prop:Pointwise}, completes the proof of \cref{thm:Maximal}.

\begin{cor}[{\cite[Prop.~4.6]{bernicot2016sharp}}]
 \label{cor:Maximal} 
 The maximal function $S^{\#}$ is bounded on
 $L^{2}$, and weak-type $(p_{0},p_{0})$.
\end{cor}

\section{Sparse Bounds}
\label{sec:Sparse}

In this section we prove \cref{t.main}.
Since $f$ and $g$ have compact support,
without loss of generality we can assume that
they are both supported on $5Q_0$
for a fixed dyadic cube $Q_0 \in \mathscr{D}$.
Here we show the existence of a sparse collection $\mathcal{S}_0$
inside $Q_0$ such that
\begin{equation*}
  \int_{Q_0} (Sf)^2  g \D{\mu}
  \lesssim \sum_{P\in\mathcal{S}_0} \left( \dashint_{5P} \lvert f \rvert^{p_0} \D{\mu} \right)^{2/p_0}  \left(\dashint_{5P} \lvert g\rvert^{q_0^*} \D{\mu} \right)^{1/q_0^*} |P|.
\end{equation*}

The enlarged cube $5Q_0$ is the union of finitely many
disjoint cubes $Q$, %
and the proof can be adapted for any of these $Q$.
Taking the union of finitely many sparse collections,
we obtain a sparse collection $\mathcal{S}$.

We will decompose our quantity in different terms:
  all will be controlled by the averages of $f$ and $g$ but
  one.
  This last term is where $f$ assumes a large value and it is similar to the original quantity
  but on a smaller scale.
  We can then iterate the decomposition, which terminates
  since the measure of the set we are decomposing shrinks
  geometrically at each iteration.

\subsection{Decomposition}
Denote by $\ell(P)$ the side length of the dyadic cube $P$.
Let us consider the (localised) dyadic version of the operator introduced in \cref{sec:Maximal},
\begin{gather*}
  \mathcal{M}_{Q_0,p_0}^*f(x) \coloneqq \sup_{\substack{P \in \mathscr{D} \\ P \subseteq Q_0}} \Big(\inf_{y \in P} \mathcal{M}_{p_0}f(y)\Big) \1_P(x) , \\
  S_{Q_0}^{*} f(x) \coloneqq \sup_{\substack{P \in \mathscr{D} \\ P \subseteq Q_0}} \Bigg( \fint_{P} \Bigg\lvert \int_{\ell(P)^2}^\infty \lvert \mathcal{Q}_t f\rvert^2 \frac{\D{t}}{t} \Bigg\rvert^{\frac{q_0}2} \D{\mu} \Bigg)^{1/q_0} \1_{P}(x) .
\end{gather*}

For a positive $\eta$ to be fixed later,
consider the set
\begin{equation*}
  E(Q_0) \coloneqq \left\{ x \in Q_0 \,:\, \max\Big\{\mathcal{M}^*_{Q_0,p_0} f(x),S_{Q_0}^{*} f(x)\Big\} > \eta \left(\fint_{5Q_0} \lvert f\rvert^{p_0} \D{\mu}\right)^{1/p_0} \right\} .
\end{equation*}
Since the operators $\mathcal{M}^*_{Q_0,p_0}$ and $S_{Q_0}^{*}$ are weak-type $(p_0,p_0)$, as shown in \cref{sec:Maximal},
there exists $\eta > 0$ such that $\lvert E(Q_0)\rvert \le \frac12
\lvert Q_0\rvert$. Decompose our form as
\begin{equation*}
  \int_{Q_0} (Sf)^2 g \D{\mu} =  \int_{Q_0 \setminus E(Q_0)} (Sf)^2 g \D{\mu} + \int_{E(Q_0)} (Sf)^2 g \D{\mu} \eqqcolon \roman{1} + \roman{2}
\end{equation*}

Term \roman{1} is controlled by using Lebesgue differentiation theorem as in \cite[Lem.~4.4]{bernicot2016sharp}
since $ \lvert Sf(x) \rvert^2  \le \lvert S^{*}_{Q_0}f(x) \rvert^2 $ for $\mu$-almost every $x$.
Thus, for $x \in Q_0\setminus E(Q_0)$ we have
\begin{align*}
  \int_{Q_0 \setminus E(Q_0)} (Sf)^2 g \D{\mu}
  & \lesssim \eta^2 \left(\fint_{5Q_0} \lvert f\rvert^{p_0} \D{\mu}\right)^{2/p_0}  \left(\fint_{Q_0} \lvert g\rvert^{q_0^*} \D{\mu}\right)^{1/q_0^*} |Q_0|.
\end{align*}

Consider term \roman{2}.
Let $\mathscr{E} \coloneqq \{P\}_{P\in\mathscr{D}}$ be a covering of $E(Q_0)$ with maximal dyadic cubes. 
Then
\begin{align*}
  \int_{E(Q_0)} (Sf)^2 g \D{\mu} & = \sum_{P\in\mathscr{E}} \int_{P}   (Sf)^2 g \D{\mu} \\
                              & = \sum_{P\in\mathscr{E}} \int_{P} \int_0^{\ell(P)^2} \lvert \mathcal{Q}_t f(x)\rvert^2 \frac{\D{t}}t g \D{\mu} + \sum_{P\in\mathscr{E}} \int_{P} \int_{\ell(P)^2}^\infty \lvert \mathcal{Q}_t f(x)\rvert^2 \frac{\D{t}}t g \D{\mu} \\
                              & \eqqcolon \roman{2}_{<} +  \roman{2}_{>} .
\end{align*}
For each $P$ in the covering, we write $f = f_{\mathsf{in}} + f_{\mathsf{out}}$, where $f_{\mathsf{in}} \coloneqq f \1_{5P}$ and $f_{\mathsf{out}} \coloneqq f \1_{(5P)^\complement}$.
Then each term in $\roman{2}_{<}$ is itself decomposed into three terms
\begin{align}
  \int_{P} \int_0^{\ell(P)^2} \lvert \mathcal{Q}_t f(x)\rvert^2 \frac{\D{t}}t g \D{\mu}
  = & \int_{P} \int_0^{\ell(P)^2} \lvert \mathcal{Q}_t f_{\mathsf{in}} \rvert^2 g \frac{\D{t}}{t} \D{\mu} \label{eq:in_term} \tag{$\roman{2}_{\mathsf{in}}$} \\
  & +    \int_{P} \int_0^{\ell(P)^2} \lvert \mathcal{Q}_t f_{\mathsf{out}} \rvert^2 g \frac{\D{t}}{t} \D{\mu} \label{eq:out_term} \tag{$\roman{2}_{\mathsf{out}}$} \\
  & + 2 \int_{P} \int_0^{\ell(P)^2} (\mathcal{Q}_t f_{\mathsf{in}})( \mathcal{Q}_tf_{\mathsf{out}} ) g \frac{\D{t}}{t} \D{\mu} \label{eq:cross_term} \tag{$\roman{2}_{\mathsf{cross}}$} .
\end{align}

Term \eqref{eq:in_term} goes into the iteration.
Terms \eqref{eq:out_term} and \eqref{eq:cross_term} are controlled by using Fubini
and applying off-diagonal estimates
as in the following lemma.

\begin{lem}\label{lemma:corona}  
  For a given dyadic cube $P$,  
  let $S_k(P) \coloneqq 2^{k+1}P\setminus 2^{k}P$ for $k\ge 2$.  
  Then for any $t >0$,
  \begin{align}\label{eq:lemma_corona_in}
    \left( \fint_{P} \lvert \mathcal{Q}_t f_{\mathsf{in}} \rvert^{q_0} \D{\mu} \right)^{1/q_0}
    & \lesssim  \left(\frac{\ell(P)}{\sqrt{t}}\right)^\nu \left( \fint_{5P}  \lvert f \rvert^{p_0} \D{\mu} \right)^{1/p_0} \\
    \left( \fint_{P} \lvert \mathcal{Q}_t f_{\mathsf{out}} \rvert^{q_0} \D{\mu} \right)^{1/q_0} \label{eq:lemma_corona_out}
    & \lesssim \left(\frac{\ell(P)}{\sqrt{t}}\right)^{-\nu-2} \sum_{k\ge 2}  2^{-k} \left( \fint_{S_k(P)}  \lvert f \rvert^{p_0} \D{\mu} \right)^{1/p_0} .
  \end{align}
\end{lem}

\begin{proof}[Proof of \cref{lemma:corona}]
  The proof follows the one in \cite[Thm.~5.7]{bernicot2016sharp}.
  For $f_{\mathsf{in}} = f \1_{5P}$,
  let $\mathcal{R}_0$ be a collection of  finite overlapping balls $R$ of radius $\sqrt{t}$
  covering $5P$.
  By linearity of the operators, the triangle inequality, off-diagonal estimates for $\mathcal{Q}_t$
  with $\rho(x)= (1 + \lvert x\rvert^2)^{-(\nu + 1)}$
  and \cref{rmk:off-diag_larger_scale} we have
  \begin{align*}
    \left( \fint_{P} \lvert \mathcal{Q}_t f_{\mathsf{in}} \rvert^{q_0} \D{\mu} \right)^{1/q_0}
    & \le \sum_{R\in\mathcal{R}_0} \left( \fint_P  \lvert \mathcal{Q}_t f \1_{R} \rvert^{q_0} \D{\mu} \right)^{1/q_0} 
    \lesssim \sum_{R\in\mathcal{R}_0} \left( \fint_R  \lvert f \rvert^{p_0} \D{\mu} \right)^{1/p_0}.
  \end{align*}
  
  Since  $5 P \subseteq \frac{15 \ell(P)}{\sqrt{t}}R$, \cref{lemma:counting_R_balls}
  implies
  \begin{equation*}
    \sum_{R\in\mathcal{R}_0} \left( \fint_R  \lvert f \rvert^{p_0} \D{\mu} \right)^{1/p_0} \lesssim \left(\frac{5 \ell(P)}{\sqrt{t}}\right)^\nu \left( \fint_{5P} \lvert f\rvert^{p_0} \D{\mu} \right)^{1/p_0} 
  \end{equation*}
  which proves \eqref{eq:lemma_corona_in}.

  For $f_{\mathsf{out}} = f\1_{(5P)^\complement}$, decompose $f$ on the squared annuli $S_k = S_k(P)$.
  Let $\mathcal{R}_k$ be the covering of $S_k$ with finite overlapping balls $R$ of radius $\sqrt{t}$.
  Linearity of the operators $\mathcal{Q}_{t}$, the triangle inequality and off-diagonal estimates for $\mathcal{Q}_t$ imply that
  \begin{align*}
    \left( \fint_P  \lvert \mathcal{Q}_t f_{\mathsf{out}} \rvert^{q_0} \D{\mu} \right)^{1/q_0}
    & \le  \sum_{k\ge 2} \sum_{R\in\mathcal{R}_k} \left( \fint_P  \lvert \mathcal{Q}_t f \1_{R} \rvert^{q_0} \D{\mu} \right)^{1/q_0}  \\
    & \lesssim  \sum_{k\ge 2} \sum_{R\in\mathcal{R}_k} \rho\Big(\frac{d(P,R)}{\sqrt{t}}\Big) \left( \fint_R  \lvert f \rvert^{p_0} \D{\mu} \right)^{1/p_0} \\
    & \lesssim  \sum_{k\ge 2} \rho\Big(\frac{d(P,S_k)}{\sqrt{t}}\Big) \sum_{R\in\mathcal{R}_k} \left( \fint_R  \lvert f \rvert^{p_0} \D{\mu} \right)^{1/p_0} \\
    & \lesssim  \sum_{k\ge 2} \rho\Big(\frac{d(P,S_k)}{\sqrt{t}}\Big) \left(\frac{2^{k+1}\ell(P)}{\sqrt{t}}\right)^\nu  \left( \fint_{S_k}  \lvert f \rvert^{p_0} \D{\mu} \right)^{1/p_0} 
  \end{align*}  
  where we used that
  the function $\rho$ is monotone decreasing
  and $d(P,R) \ge d(P,S_k)$.
  The last inequality follows by applying \cref{lemma:counting_R_balls},
  since
  $S_k(P) \subseteq 2^{k} P \subseteq \frac{2^{k+1} \ell(P)}{\sqrt{t}}R$.
  
  Finally, we have enough decay from the remaining product, since
  \begin{equation*}%
    \rho\Big(\frac{d(P,S_k)}{\sqrt{t}}\Big)  \left(\frac{2^{k+1}\ell(P)}{\sqrt{t}}\right)^\nu \lesssim \Big( \frac{2^k \ell(P)}{\sqrt{t}}\Big)^{-\nu-2}
  \end{equation*}
  This follows because $ d(P,S_k) = d(P,2^{k+1}P\setminus 2^k P)$ is comparable with $2^k \ell(P)$
  and the function $\rho(x)= (1 + \lvert x\rvert^2)^{-(\nu + 1)}$ decays faster than $x^\nu$
  for $x \gg 1$.
  This proves estimate \eqref{eq:lemma_corona_out}.
\end{proof}

  We will use \cref{lemma:corona} to control the different terms left in the decomposition.
  \begin{rmk}\label{remark:control_geometric_sum}
    The geometric sum in  \eqref{eq:lemma_corona_out} is controlled using the stopping condition:
    the integral over $S_k$ is bounded by  the integral over the ball $2^{k+1}P$, so
    \begin{align*}
      \left( \sum_{k\ge 2} 2^{-k}\left( \fint_{S_k}  \lvert f \rvert^{p_0} \D{\mu} \right)^{1/p_0} \right)^2 
      & \lesssim \left( \sup_{k\ge 2} \left( \fint_{2^{k+1}P}  \lvert f \rvert^{p_0} \D{\mu} \right)^{1/p_0} \right)^2 \\
      & \lesssim \Big( \inf_{\substack{y\in P^a \\ P^a \text{ parent of } P}} \mathcal{M}_{p_0}f(y) \Big)^2 \\
      & \lesssim \eta^2 \left( \fint_{5Q_0} \lvert f \rvert^{p_0} \D{\mu} \right)^{2/p_0},
    \end{align*}
    where we used that $P$ is a maximal cube covering $E$.
    Similarly for the average on $5P$:
    \begin{align*}
      \left( \fint_{5P}  \lvert f \rvert^{p_0} \D{\mu} \right)^{2/p_0}
      \lesssim \Big( \inf_{y\in P^a} \mathcal{M}_{p_0}f(y) \Big)^2
      \lesssim  \eta^2 \left( \fint_{5Q_0}  \lvert f \rvert^{p_0} \D{\mu} \right)^{2/p_0} .
    \end{align*}
  \end{rmk}

  \begin{rmk}[Control on the $q_0^*$-average of $g$]
    As in \cite{bernicot2016sharp}
    the average of $g$ is controlled by the maximal function
    in a  similar fashion:
    \begin{equation*}
      \left( \fint_{P} \lvert g\rvert^{q_0^*} \D{\mu} \right)^{1/q_0^*} \le \inf_{x \in P} \mathcal{M}_{q_0^*} g(x) .
    \end{equation*}
    Summing over all cubes $P$ in $\mathscr{E}$ and using Kolmogorov's lemma \cite[Lem.~5.16]{Duo}
    we obtain
    \begin{equation}\label{eq:bound_on_g}
      \sum_{P}  \left( \fint_{P} \lvert g\rvert^{q_0^*} \D{\mu} \right)^{1/q_0^*} \abs{P} \le \int_{E(Q_0)} \mathcal{M}_{q_0^*} g(x) \D{\mu}
      \lesssim \abs{Q_0} \left( \fint_{5Q_0} \lvert g\rvert^{q_0^*} \D{\mu} \right)^{1/q_0^*} .
    \end{equation}
  \end{rmk}
  
  \subsection{Out term}
  Consider  \eqref{eq:out_term}.
  Applying Fubini and Hölder's inequality, we have
  \begin{equation*}
    \int_{P} \int_0^{\ell(P)^2} \lvert \mathcal{Q}_t f_{\mathsf{out}} \rvert^2 g \frac{\D{t}}{t} \D{\mu}
    \le \int_0^{\ell(P)^2}  \left( \fint_{P} \lvert \mathcal{Q}_t f_{\mathsf{out}} \rvert^{q_0} \D{\mu} \right)^{2/q_0} \frac{\D{t}}{t} \left(\fint_P \lvert g\rvert^{q_0^*} \D{\mu} \right)^{1/q_0^*} \abs{P}.
  \end{equation*}
  The average of $g$ is controlled as in \eqref{eq:bound_on_g}.
  Apply \cref{lemma:corona}
  to the first factor:
  \begin{align*}
    \int_0^{\ell(P)^2} \left( \fint_{P} \lvert \mathcal{Q}_t f_{\mathsf{out}}\rvert^{q_0} \D{\mu} \right)^{2/q_0}  \frac{\D{t}}t 
    & \lesssim  \int_0^{\ell(P)^2} \Bigg( \sum_{k\ge 2}  \frac{\sqrt{t}}{2^k \ell(P)} \left( \fint_{S_k}  \lvert f \rvert^{p_0} \D{\mu} \right)^{1/p_0} \Bigg)^2 \frac{\D{t}}t \\
    & \lesssim \Bigg( \sum_{k\ge 2} 2^{-k} \left( \fint_{S_k}  \lvert f \rvert^{p_0} \D{\mu} \right)^{1/p_0} \Bigg)^2
  \end{align*}
  which is controlled as in \cref{remark:control_geometric_sum}. This case is concluded.

  \subsection{Cross term}
  Consider \eqref{eq:cross_term}.
  We exchange the integrals, then an application of Hölder's and Cauchy--Schwarz inequality give
  \begin{align*}
    \int_{P} \int_0^{\ell(P)^2} & (\mathcal{Q}_t f_{\mathsf{in}})( \mathcal{Q}_tf_{\mathsf{out}} ) g \frac{\D{t}}{t} \D{\mu} \\
                             & \le \int_0^{\ell(P)^2} \left( \fint_{P}  \lvert (\mathcal{Q}_t f_{\mathsf{in}})( \mathcal{Q}_tf_{\mathsf{out}} ) \rvert^{q_0/2} \D{\mu} \right)^{2/q_0} \frac{\D{t}}{t}  \left( \fint_P \lvert g\rvert^{q_0^*} \D{\mu} \right)^{1/q_0^*}  \abs{P} \\
                             & \le \int_0^{\ell(P)^2} \left( \fint_{P}  \lvert \mathcal{Q}_t f_{\mathsf{in}}\rvert^{q_0} \D{\mu} \right)^{1/q_0} \left(\fint_P \lvert \mathcal{Q}_tf_{\mathsf{out}} \rvert^{q_0} \D{\mu} \right)^{1/q_0}  \frac{\D{t}}{t}
                               \left( \fint_P \abs{g}^{q_0^{*}} \D{\mu} \right)^{1/q_0^{*}} \abs{P}.
  \end{align*}
  For the first factor with $f_{\mathsf{in}}$, off-diagonal estimates as in \cref{lemma:corona} imply
  \begin{equation}\label{eq:cross_term_in}
    \left( \fint_{P}  \lvert \mathcal{Q}_t f_{\mathsf{in}}\rvert^{q_0} \D{\mu} \right)^{1/q_0} \lesssim \left( \frac{\ell(P)}{\sqrt{t}} \right)^{\nu} \left( \fint_{5P} \lvert f\rvert^{p_0} \D{\mu} \right)^{1/p_0} .
  \end{equation}
  For the second factor with $f_{\mathsf{out}}$, \cref{lemma:corona} implies
  \begin{equation}\label{eq:cross_term_out}
    \left(\fint_P \lvert \mathcal{Q}_t f_{\mathsf{out}} \rvert^{q_0} \D{\mu} \right)^{1/q_0} \lesssim \left(\frac{\ell(P)}{\sqrt{t}}\right)^{-\nu-2} \sum_{k\ge 2}  2^{-k} \left( \fint_{S_k(P)}  \lvert f \rvert^{p_0} \D{\mu} \right)^{1/p_0} .
  \end{equation}
  
  Combining \eqref{eq:cross_term_in} and \eqref{eq:cross_term_out} gives
  \begin{align*}
    \int_0^{\ell(P)^2} & \left( \fint_{P}  \lvert \mathcal{Q}_t f_{\mathsf{in}}\rvert^{q_0} \D{\mu} \right)^{1/q_0} \left(\fint_P \lvert \mathcal{Q}_tf_{\mathsf{out}} \rvert^{q_0} \D{\mu} \right)^{1/q_0}  \frac{\D{t}}{t} \\
                     & \lesssim \int_0^{\ell(P)^2}  \left(\frac{\sqrt{t}}{\ell(P)}\right)^{2}  \frac{\D{t}}{t} \left( \fint_{5P} \lvert f\rvert^{p_0} \D{\mu} \right)^{1/p_0}  \sum_{k\ge 2}  2^{-k} \left( \fint_{S_k(P)}  \lvert f \rvert^{p_0} \D{\mu} \right)^{1/p_0} \\
                     & \lesssim \eta^2 \left( \fint_{5Q_0} \lvert f\rvert^{p_0} \D{\mu} \right)^{2/p_0}
  \end{align*}
  where the last estimate 
  follows as in \cref{remark:control_geometric_sum}.

\subsection{Large scales}
Consider $\roman{2}_{>}$. %
Let $P^a$ be the dyadic parent of $P$, so that $\ell(P^a) = 2 \ell(P)$. Then
\begin{align}\label{eq:large_splitting}
  \int_{P} & \int_{\ell(P)^2}^\infty \lvert \mathcal{Q}_t f(x)\rvert^2 \frac{\D{t}}t g \D{\mu}  \nonumber \\
  = &
  \int_{P} \int_{\ell(P)^2}^{\ell(P^a)^2} \lvert \mathcal{Q}_t f(x)\rvert^2 \frac{\D{t}}t g \D{\mu}
  + \int_{P} \int_{\ell(P^a)^2}^\infty \lvert \mathcal{Q}_t f(x)\rvert^2 \frac{\D{t}}t g \D{\mu} .
\end{align}

In the first term, we exchange the integrals
and apply Hölder's inequality %
\begin{align*}
  \int_{\ell(P)^2}^{\ell(P^a)^2} \int_{P} \lvert \mathcal{Q}_t f(x)\rvert^2  g \D{\mu} \frac{\D{t}}t  
  & \le \int_{\ell(P)^2}^{\ell(P^a)^2}  \left( \fint_{P} \lvert \mathcal{Q}_t f(x)\rvert^{q_0} \D{\mu} \right)^{2/q_0}  \frac{\D{t}}t
    \left( \fint_{P} \lvert g\rvert^{q_0^*} \D{\mu} \right)^{1/q_0^*} \abs{P} .
\end{align*}

Applying \cref{lemma:corona} and
using that $\sqrt{t}$ is comparable with $\ell(P)$, we obtain
\begin{align*}
  \int_{\ell(P)^2}^{\ell(P^a)^2} & \left( \fint_{P} \lvert \mathcal{Q}_t f\rvert^{q_0} \D{\mu} \right)^{2/q_0}  \frac{\D{t}}t \\
                              & \lesssim  \int_{\ell(P)^2}^{\ell(P^a)^2} \left( \left(\frac{\ell(P)}{\sqrt{t}}\right)^\nu \left( \fint_{5P}  \lvert f \rvert^{p_0} \D{\mu} \right)^{1/p_0}
                                + \sum_{k\ge 2}  \frac{\sqrt{t}}{2^k \ell(P)} \left( \fint_{S_k}  \lvert f \rvert^{p_0} \D{\mu} \right)^{1/p_0}\right)^2 \frac{\D{t}}t \\
                              & \lesssim \left( \left( \fint_{5P}  \lvert f \rvert^{p_0} \D{\mu} \right)^{1/p_0} + \sum_{k\ge 2} 2^{-k} \left( \fint_{S_k}  \lvert f \rvert^{p_0} \D{\mu} \right)^{1/p_0} \right)^2,
\end{align*}
which again is controlled as in \cref{remark:control_geometric_sum}.
The average of $g$ is estimated as in \eqref{eq:bound_on_g}.

The second term in \eqref{eq:large_splitting},
after applying Hölder's inequality, is controlled by the maximal truncation 
\begin{align*}
  \int_{P} \int_{\ell(P^a)^2}^\infty \lvert \mathcal{Q}_t f(x)\rvert^2 \frac{\D{t}}t g \D{\mu} & \le
  \left( \fint_{P} \Big( \int_{\ell(P^a)^2}^\infty \lvert \mathcal{Q}_t f(x)\rvert^2 \frac{\D{t}}t \Big)^{q_0/2} \D{\mu} \right)^{2/q_0}
  \left( \fint_{P} \lvert g\rvert^{q_0^*} \D{\mu} \right)^{1/q_0^*} \abs{P} \\
  & \lesssim \inf_{x \in P^a} (S_{Q_0}^{*} f)^2(x) \left( \fint_{P} \lvert g\rvert^{q_0^*} \D{\mu} \right)^{1/q_0^*} \abs{P} \\
  & \lesssim \eta^2 \left(\fint_{5Q_0}\lvert f\rvert^{p_0} \D{\mu}\right)^{2/p_0} \left( \fint_{P} \lvert g\rvert^{q_0^*} \D{\mu} \right)^{1/q_0^*} \abs{P} .
\end{align*}

We have shown that
\begin{align*}
  \int_{Q_0} \int_0^{\infty} \lvert \mathcal{Q}_t f \rvert^{2} g \frac{\D{t}}{t} \D{\mu} \lesssim \eta^2 & \left( \fint_{5Q_0} \abs{f}^{p_0} \D{\mu} \right)^{2/p_0} \left( \fint_{5Q_0} \abs{g}^{q_0^{*}} \D{\mu} \right)^{1/q_0^{*}} \abs{Q_0} \\
  & + \sum_{P} \int_{P} \int_0^{\ell(P)^2} \lvert \mathcal{Q}_t f\1_{5P}\rvert^2 \frac{\D{t}}t g \D{\mu} .
\end{align*}

Let $\mathcal{S} = \{Q_0\}$.
We add all $P$ in the sum to $\mathcal{S}$
and we repeat the argument on each term in the sum.
This iteration gives the desired bound: a sum of averages of $f$ and $g$
on cubes in the collection $\mathcal{S}$.
We can choose $\eta>0$ such that
$\lvert E(Q)\rvert \le \frac12 \lvert Q\rvert$ for each $Q\in\mathcal{S}$.
Then $\mathcal{S}$ is sparse since each $Q\in\mathcal{S}$
has a subset $F_Q \coloneqq Q \setminus E(Q)$ with the property that
$\{ F_Q \}_{Q\in\mathcal{S}}$ is a disjoint family
and $\lvert F_Q\rvert > \frac12 \lvert Q\rvert$ by construction. \qed

\section{Weighted Boundedness}
\label{sec:Weighted}

In this section, we provide the proof of \cref{thm:weight_bounds}.
We begin by recalling the notation $p^{*} := (p/2)' = \frac{p}{p - 2}$
for $p > 2$.
We will also make use of the notation
$$
  \phi(p) := \br{\frac{q_{0}}{p}}' \br{\frac{p}{p_{0}} - 1} + 1
  $$
for $1 \leq p_{0} < 2 < q_{0} \leq \infty$ and $p \in
  (p_{0},q_{0})$, which was previously introduced in \cref{subsec:weight_classes}.

  \begin{rmk}
    \label{rmk:CriticalExponent}
  Define the critical index $\mathfrak{p}$ through
  \begin{equation}
    \label{eqtn:MFrakP}
  \mathfrak{p} := 2 + p_{0} - \frac{2 p_{0}}{q_{0}} = 2 + \frac{p_{0}}{q_{0}^{*}}.
\end{equation}
The critical exponent is the unique $p \in (1,\infty)$ that satisfies the
relation $p^{*} = \phi(p)$.
  It is easy to check that $\mathfrak{p}$ is contained in the interval
  $(2,q_{0})$ and that it satisfies the relation
  \begin{equation}
    \label{eqtn:Critical}
    \frac{1}{\mathfrak{p} - p_{0}} = \br{\frac{q_{0}}{\mathfrak{p}}}'
    \frac{1}{2 q_{0}^{*}}.
  \end{equation}
 Thus, we also have that
$$
\gamma(p) = \max \br{\frac{1}{p - p_{0}},\br{\frac{q_{0}}{p}}'
  \frac{1}{2 q_{0}^{*}}} = \br{\frac{q_{0}}{p}}' \frac{1}{2 q_{0}^{*}}
$$
if and only if $p \geq \mathfrak{p}$, and $\gamma(p) = (p -
p_{0})^{-1}$ if and only if $p \leq \mathfrak{p}$.
  In \cite{bernicot2016sharp},
  the critical exponent for 
  the linear sparse domination
  is $1 + p_0/q_0'$ which is evidently analogous to
  the definition of $\mathfrak{p}$ in \eqref{eqtn:MFrakP}.
\end{rmk}

\subsection{Proof of Theorem \ref{thm:weight_bounds}}
Fix $p \in (2,q_{0})$.  Notice that by \eqref{eqtn:WeightProperty},
$$
\brs{w^{(\frac{q_{0}}{p})'}}_{A_{\phi(p)}} \leq
\br{\brs{w}_{A_{\frac{p}{p_{0}}}} \cdot \brs{w}_{RH_{(\frac{q_{0}}{p})'}}}^{(\frac{q_{0}}{p})'}.
$$
This tells us that in order to prove estimate \eqref{e.secondmain}, it
is sufficient to demonstrate the stronger estimate
\begin{equation}
  \begin{aligned}
    \label{e.strongsecondmain}
    \sum_{P \in \mathcal{S}} \left(\dashint_{5 P}
      \abs{f}^{p_0}\D{\mu}\right)^{2/p_0}  \left(\dashint_{5 P}  \abs{g}^{q_{0}^{*}} \D{\mu} \right)^{1/q_{0}^{*}} \abs{P}
    \leq C_0 \brs{w^{(\frac{q_{0}}{p})'}}_{A_{\phi(p)}}^{\frac{2 \gamma(p)}{(q_{0}/p)'}} \norm{f}_{L^{p}(w)}^{2} \norm{g}_{L^{p^{*}}(\sigma)}.
    \end{aligned}
  \end{equation}
By Theorem \ref{thm:Dyadic}, for each $P \in
\mathcal{S}$ there will exist $\bar{P} \in \mathscr{D}$ for which $5 P \subset \bar{P}$ and $\abs{\bar{P}} \lesssim \abs{5
P}$. Then $\abs{\bar{P}} \lesssim \abs{P}$ by the doubling property of
dyadic cubes. As the collection $\mathcal{S}$ is sparse, there must
exist a collection of disjoint sets $\lb E_{P} \rb_{P \in \mathcal{S}}$
such that $E_{P} \subset P$ and $\abs{P} \lesssim \abs{E_{P}}$ for all
$P \in \mathcal{S}$. We therefore have
$$
\abs{\bar{P}} \lesssim \abs{E_{P}}.
$$
Define the weight $v := w^{(q_{0}/p)'}$ and $r :=
\phi(p) =
\br{\frac{q_{0}}{p}}' \br{\frac{p}{p_{0}} - 1} +
1$. Set $u$ to be the dual weight of $v$ in $A_{r}$, $u := v^{1 -
  r'}$.
 We have
\begin{equation*}
 \br{\dashint_{\bar{P}} \abs{f}^{p_{0}} \D{\mu}}^{\frac{2}{p_{0}}}= \br{\frac{1}{u(\bar{P})} \int_{\bar{P}} \abs{f u^{-\frac{1}{p_{0}}}}^{p_{0}} u
 \D{\mu}}^{\frac{2}{p_{0}}} \br{\dashint_{\bar{P}} u \D{\mu}}^{\frac{2}{p_{0}}}
\end{equation*}
and
\begin{equation*}
 \br{
   \dashint_{\bar{P}} \abs{g}^{q_{0}^{*}} \D{\mu}}^{\frac{1}{q_{0}^{*}}} \\
 =\br{\frac{1}{v(\bar{P})} \int_{\bar{P}} \abs{g
     v^{-\frac{1}{q_{0}^{*}}}}^{q_{0}^{*}} v \D{\mu}}^{\frac{1}{q_{0}^{*}}} \br{\dashint_{\bar{P}}
   v \D{\mu}}^{\frac{1}{q_{0}^{*}}}. 
\end{equation*}
Applying these two relations to our sparse form leads to
\begin{equation}
  \begin{aligned}
    \sum_{P \in \mathcal{S}} \left(\dashint_{5 P} \abs{f}^{p_0}
      \D{\mu}\right)^{2/p_0} & \left(\dashint_{5 P}
      \abs{g}^{q_{0}^{*}} \D{\mu}\right)^{1/q_{0}^{*}} \abs{P} \\
    &\lesssim \sum_{P \in \mathcal{S}} \br{\dashint_{\bar{P}}
      \abs{f}^{p_{0}} \D{\mu}}^{\frac{2}{p_{0}}}
    \br{\dashint_{\bar{P}} \abs{g}^{q_{0}^{*}} \,
      d\mu}^{\frac{1}{q_{0}^{*}}} \abs{P} \\
    &= \sum_{P \in \mathcal{S}} \br{\frac{1}{u(\bar{P})} \int_{\bar{P}} \abs{f u^{-\frac{1}{p_{0}}}}^{p_{0}} u \D{\mu}}^{\frac{2}{p_{0}}}
    \br{\frac{1}{v(\bar{P})} \int_{\bar{P}} \abs{g v^{-\frac{1}{q_{0}^{*}}}}^{q_{0}^{*}} v \D{\mu}}^{\frac{1}{q_{0}^{*}}} \\
    & \qquad \qquad \qquad \cdot \br{\dashint_{\bar{P}} v \D{\mu}}^{\frac{1}{q_{0}^{*}}}  \br{\dashint_{\bar{P}} u \D{\mu}}^{\frac{2}{p_{0}}} \abs{P}.                              
  \end{aligned}\label{eqtn:Bounded2.5}
\end{equation}

\textit{Case 1: $p \geq \mathfrak{p}$.} Note that by Remark
\ref{rmk:CriticalExponent} this assumption is
equivalent to assuming that $\gamma(p) = \br{\frac{q_{0}}{p}}'
\frac{1}{2 q_{0}^{*}}$. If we define
$$
\kappa(p) := \frac{2}{p_{0}} - \frac{r - 1}{q_{0}^{*}},
$$
then our assumption is also equivalent to the condition $\kappa(p)
\leq 0$.
The fact that $u$ is the conjugate weight of $v$ in $A_{r}$
implies that for $P \in \mathcal{S}$,
\begin{align*}\begin{split}  
   \br{\dashint_{\bar{P}} v \D{\mu}}^{\frac{1}{q_{0}^{*}}} \br{\dashint_{\bar{P}} u \D{\mu}}^{\frac{2}{p_{0}}}
  &= \br{\dashint_{\bar{P}} v \D{\mu}}^{\frac{1}{q_{0}^{*}}}
    \br{\dashint_{\bar{P}} u \D{\mu}}^{\frac{r - 1}{q_{0}^{*}}}
    \br{\dashint_{\bar{P}} u \D{\mu}}^{\kappa(p)} \\
  &\le \brs{v}_{A_{r}}^{1/q_{0}^{*}}\br{\dashint_{\bar{P}} u \D{\mu}}^{\kappa(p)}.   
 \end{split}\end{align*}
This estimate can be applied to \eqref{eqtn:Bounded2.5} to produce
\begin{equation}\begin{aligned}
   \label{eqtn:Bounded2.51}
  \sum_{P \in \mathcal{S}}& \left(\dashint_{5 P} \abs{f}^{p_0}
    \D{\mu}\right)^{2/p_0}   \left(\dashint_{5 P} \abs{g}^{q_{0}^{*}} \D{\mu} \right)^{1/q_{0}^{*}}  \lvert P \rvert \\
  & \lesssim \brs{v}_{A_{r}}^{\frac{1}{q_{0}^{*}}} \sum_{P \in \mathcal{S}} \br{\frac{1}{u(\bar{P})} \int_{\bar{P}} \abs{f u^{-\frac{1}{p_{0}}}}^{p_{0}} u
    \D{\mu}}^{\frac{2}{p_{0}}} \br{\frac{1}{v(\bar{P})} \int_{\bar{P}} \abs{g
      v^{-\frac{1}{q_{0}^{*}}}}^{q_{0}^{*}} v
    \D{\mu}}^{\frac{1}{q_{0}^{*}}} \br{\dashint_{\bar{P}} u
    \D{\mu}}^{\kappa(p)} \abs{P}.
\end{aligned}
\end{equation}
Since $\abs{\bar{P}} \lesssim \abs{E_{P}}$ and $\kappa(p) \leq 0$,
\begin{align*}
  \br{\dashint_{\bar{P}} u \D{\mu}}^{\kappa(p)} \abs{P}
  &\lesssim \br{\dashint_{E_{P}} u \D{\mu}}^{\kappa(p)} \abs{E_{P}} \\
  &=  u(E_{P})^{\kappa(p)} \abs{E_{P}}^{1 - \kappa(p)} .
\end{align*}
For $\lambda := (1 - \kappa(p))^{-1}$ notice that
\begin{equation*}
  \frac{\lambda}{p/2} + \frac{\lambda}{p^{*}} - \lambda \kappa(p) =
  \lambda(1 - \kappa(p)) = 1.
\end{equation*}
Also, it is straightforward to check by substituting in the definition
$u \coloneqq v^{1 - r'}$ that the constant function $1$ can be decomposed as
$$
1 = u^{\frac{1}{p/2}} v^{\frac{1}{p^{*}}} u^{- \kappa(p)}.
$$
From this, H\"{o}lder's inequality implies
\begin{align*}\begin{split}  
 \abs{E_{P}} &= \int_{E_{P}} u^{\frac{\lambda}{p/2}}
 v^{\frac{\lambda}{p^{*}}} u^{- \lambda \kappa(p)} \D{\mu} \\ &\leq \br{\int_{E_{P}} u}^{\frac{\lambda}{p/2}}
 \br{\int_{E_{P}} v}^{\frac{\lambda}{p^{*}}} \br{\int_{E_{P}} u}^{-\lambda \kappa(p)},
\end{split}\end{align*}
and therefore, raising to the power $1/\lambda$ produces
\begin{equation*}
  u(E_P)^{\kappa(p)} \abs{E_P}^{1/\lambda} \le  u(E_P)^{2/p} v(E_P)^{1/p^*}.
\end{equation*}

Applying this estimate to \eqref{eqtn:Bounded2.51}
and H\"{o}lder's inequality leads to
\begin{align*}
  \sum_{P \in \mathcal{S}}& \left(\dashint_{5 P} \abs{f}^{p_0}
                            \D{\mu}\right)^{2/p_0}   \left(\dashint_{5
                            P} \abs{g}^{q_{0}^{*}} \D{\mu} \right)^{1/q_{0}^{*}}  \lvert P \rvert \\
  & \lesssim \brs{v}_{A_{r}}^{\frac{1}{q_{0}^{*}}} \sum_{P \in \mathcal{S}} \br{\frac{1}{u(\bar{P})} \int_{\bar{P}} \abs{f u^{-\frac{1}{p_{0}}}}^{p_{0}} u
    \D{\mu}}^{\frac{2}{p_{0}}} \br{\frac{1}{v(\bar{P})} \int_{\bar{P}} \abs{g
      v^{-\frac{1}{q_{0}^{*}}}}^{q_{0}^{*}} v
    \D{\mu}}^{\frac{1}{q_{0}^{*}}} u(E_{P})^{\frac{2}{p}} v(E_{P})^{\frac{1}{p^{*}}} \\
  &\leq \brs{v}_{A_{r}}^{\frac{1}{q_{0}^{*}}} \sum_{P \in \mathcal{S}}
    \br{\int_{E_{P}} \dyadicM_{p_{0},u} (f u^{-\frac{1}{p_{0}}})^{p} u \D{\mu}}^{\frac{2}{p}} \br{\int_{E_{P}} \dyadicM_{q_{0}^{*},v}(g
    v^{-\frac{1}{q_{0}^{*}}})^{p^{*}} v \D{\mu}}^{\frac{1}{p^{*}}} \\
  & \le \brs{v}_{A_{r}}^{\frac{1}{q_{0}^{*}}} \br{\sum_{P \in
    \mathcal{S}} \int_{E_{P}} \dyadicM_{p_{0},u}(f
    u^{-\frac{1}{p_{0}}})^{p} u \D{\mu}}^{\frac{2}{p}} \br{\sum_{P \in
    \mathcal{S}} \int_{E_{P}} \dyadicM_{q_{0}^{*},v}(g
    v^{-\frac{1}{q_{0}^{*}}})^{p^{*}} v \D{\mu}}^{\frac{1}{p^{*}}} \\
  &\leq \brs{v}_{A_{r}}^{\frac{1}{q_{0}^{*}}} \br{ \int \dyadicM_{p_{0},u}(f
    u^{-\frac{1}{p_{0}}})^{p} u \D{\mu}}^{\frac{2}{p}} \br{\int \dyadicM_{q_{0}^{*},v}(g
    v^{-\frac{1}{q_{0}^{*}}})^{p^{*}} v \D{\mu}}^{\frac{1}{p^{*}}}.
\end{align*}

    Since $p > p_{0}$ the operator $\dyadicM_{p_0,u}$ is bounded on
    $L^{p}(u \D{\mu})$ with constant independent of $u$. Similarly, since $p^{*} >
    q_{0}^{*}$ the operator $\dyadicM_{q_{0}^{*},v}$ is bounded on $L^{p^{*}}(v
    \D{\mu})$ with constant independent of $v$. These observations
     lead to the estimate
    \begin{equation*}
     \sum_{P \in \mathcal{S}} \left(\dashint_{5 P}
       \abs{f}^{p_0}\D{\mu}\right)^{2/p_0}  \left(\dashint_{5 P}  \abs{g}^{q_{0}^{*}}\D{\mu}\right)^{1/q_{0}^{*}} \abs{P} \lesssim \brs{v}_{A_{r}}^{\frac{1}{q_{0}^{*}}}
      \br{\int \abs{f}^{p} u^{1 -
          \frac{p}{p_{0}}} \D{\mu}}^{\frac{2}{p}}
      \br{\int \abs{g}^{p^{*}} v^{1 -
          \frac{p^{*}}{q_{0}^{*}}} \D{\mu}}.
    \end{equation*}
    From this estimate and the relation $\frac{1}{q_{0}^{*}} = \frac{2
    \gamma(p)}{(q_{0}/p)'}$ it is clear that \eqref{e.strongsecondmain} will
    follow if it can be shown that $u^{1 -
      \frac{p}{p_{0}}} = w$ and $v^{1 -
      \frac{p^{*}}{q_{0}^{*}}} = \sigma$. Let's first prove
    that $u^{1 - \frac{p}{p_{0}}} = w$. As $u$ is defined
    through $u = v^{1 - r'} = w^{\br{\frac{q_{0}}{p}}'(1 - r')}$ we have
    \begin{equation*}
      u^{1 - \frac{p}{p_{0}}} = w^{\br{\frac{q_0}{p}}'(1
        - r') \br{1 - \frac{p}{p_{0}}}}
      = w^{(1 - r')(1-r)} = w,
    \end{equation*}
    where the second equality follows from the definition $r = \phi(p)$.

It remains to show that $v^{1 -
      \frac{p^{*}}{q_{0}^{*}}} = \sigma$. The
    definitions $v = w^{\br{\frac{q_{0}}{p}}'}$ and $\sigma
    = w^{1 - p^{*}}$ transform this relation into
    $$
w^{\br{\frac{q_0}{p}}' \br{1 -
    \frac{p^{*}}{q_{0}^{*}}}} = w^{1 - p^{*}}.
$$
It must therefore be proved that
\begin{equation*}
  \br{\frac{q_0}{p}}' \br{1 -
    \frac{p^{*}}{q_{0}^{*}}} = 1 - p^{*}.
\end{equation*}
This is equivalent to showing that
$$
 \br{\frac{q_{0}}{q_{0} -
    p}} \br{1 - \frac{p(q_{0} - 2)}{q_{0}(p - 2)}} = 1 - \frac{p}{p - 2}.
$$
Through simple algebraic manipulation, it is easy to check that the
two sides of the above equality indeed coincide. This validates the relation $v^{1 -
  \frac{p^{*}}{q_{0}^{*}}} = \sigma$ and completes our
proof for $p \geq \mathfrak{p}$.

\vspace*{0.2in}

\textit{Case 2: $p \leq \mathfrak{p}$}. This assumption is equivalent to assuming
that $\gamma(p) = \frac{1}{p - p_{0}}$ or, alternatively, $\kappa(p) \geq 0$. Define
$$
\bar{\kappa}(p) := \frac{1}{q_{0}^{*}} - \frac{2}{p_{0}(r - 1)}.
$$
Then
\begin{align*}
  \br{\dashint_{\bar{P}} v \D{\mu}}^{\frac{1}{q_{0}^{*}}} \br{\dashint_{\bar{P}} u \D{\mu}}^{\frac{2}{p_{0}}}
  &= \br{\dashint_{\bar{P}} v \D{\mu}}^{\frac{2}{p_{0}(r - 1)}}
    \br{\dashint_{\bar{P}} u \D{\mu}}^{\frac{2}{p_{0}}} \br{\dashint_{\bar{P}} v \D{\mu}}^{\bar{\kappa}(p)}
  \\
  &\leq \brs{v}_{A_{r}}^{\frac{2}{p_{0}(r - 1)}} \br{\dashint_{\bar{P}} v \D{\mu}}^{\bar{\kappa}(p)}.
\end{align*}
Combining this with \eqref{eqtn:Bounded2.5} gives
\begin{equation}
  \begin{aligned}
    \label{eqtn:Bounded6}
    \sum_{P \in \mathcal{S}} &\br{\dashint_{5 P} \abs{f}^{p_{0}} \D{\mu}}^{\frac{2}{p_{0}}} \br{\dashint_{5 P} \abs{g}^{q_{0}^{*}} \D{\mu}}^{\frac{1}{q_{0}^{*}}} \abs{P} \\
    &\lesssim \brs{v}_{A_{r}}^{\frac{2}{p_{0}(r - 1)}} \sum_{P \in \mathcal{S}} \br{\frac{1}{u(\bar{P})} \int_{\bar{P}} \abs{f
        u^{-\frac{1}{p_{0}}}}^{p_{0}} u \D{\mu}}^{\frac{2}{p_{0}}}
    \br{\frac{1}{v(\bar{P})} \int_{\bar{P}} \abs{g
        v^{-\frac{1}{q_{0}^{*}}}}^{q_{0}^{*}} v \D{\mu}}^{\frac{1}{q_{0}^{*}}} \br{\dashint_{\bar{P}} v \D{\mu}}^{\bar{\kappa}(p)} \abs{P}.
    \end{aligned}
  \end{equation}
It is clear that $\bar{\kappa}(p) = - (r - 1)^{-1}\kappa(p) \leq 0$. It then
follows from this and $\abs{\bar{P}} \lesssim \abs{E_{P}}$ that
\begin{align*}
  \br{\dashint_{\bar{P}} v \D{\mu}}^{\bar{\kappa}(p)} \abs{P}
  & \lesssim \br{\dashint_{E_{P}} v \D{\mu}}^{\bar{\kappa}(p)} \abs{E_{P}} \\
  & = v(E_{P})^{\bar{\kappa}(p)} \abs{E_{P}}^{1 - \bar{\kappa}(p)} .
\end{align*}

Define $\bar{\lambda} := (1 - \bar{\kappa}(p))^{-1}$. Then we have
$$
\frac{\bar{\lambda}}{p/2} + \frac{\bar{\lambda}}{p^{*}} -
\bar{\lambda} \bar{\kappa}(p) = \bar{\lambda}(1 - \bar{\kappa}(p)) = 1.
$$
Also, it is straightforward to check by substituting in the definition
$u = v^{1 - r'}$ that the constant function $1$ can be decomposed as
$$
1 = u^{\frac{1}{p/2}} v^{\frac{1}{p^{*}}} v^{- \bar{\kappa}(p)}.
$$
H\"{o}lder's inequality then implies
\begin{align*}\begin{split}  
 \abs{E_{P}} &= \int_{E_{P}} u^{\frac{\bar{\lambda}}{p/2}}
 v^{\frac{\bar{\lambda}}{p^{*}}} v^{- \bar{\lambda} \bar{\kappa}(p)}
 \D{\mu} \\
 &\leq \br{\int_{E_{P}} u}^{\frac{\bar{\lambda}}{p/2}}
 \br{\int_{E_{P}} v}^{\frac{\bar{\lambda}}{p^{*}}} \br{\int_{E_{P}}
   v}^{- \bar{\lambda} \bar{\kappa}(p)}.
\end{split}\end{align*}
Raising to the power $1/\bar{\lambda}$ produces
\begin{equation*}
  v(E_P)^{\bar{\kappa}(p)} \abs{E_P}^{1/\bar{\lambda}} \le   u(E_P)^{2/p} v(E_P)^{1/p^*}.
\end{equation*}
Applying this to \eqref{eqtn:Bounded6}, since $1/\bar{\lambda} = 1 - \bar{\kappa}(p)$, yields
\begin{align*}\begin{split}  
    \sum_{P \in \mathcal{S}} &\br{\dashint_{5 P} \abs{f}^{p_{0}} \D{\mu}}^{\frac{2}{p_{0}}}
    \br{\dashint_{5 P} \abs{g}^{q_{0}^{*}} \D{\mu}}^{\frac{1}{q_{0}^{*}}} \abs{P} \\ & \lesssim
 \brs{v}_{A_{r}}^{\frac{2}{p_{0}(r - 1)}} \sum_{P \in \mathcal{S}}
 \br{\frac{1}{u(\bar{P})} \int_{\bar{P}} \abs{f u^{-\frac{1}{p_{0}}}}^{p_{0}} u \D{\mu}}^{\frac{2}{p_{0}}} \br{\frac{1}{v(\bar{P})} \int_{\bar{P}} \abs{g
   v^{-\frac{1}{q_{0}^{*}}}}^{q_{0}^{*}} v \D{\mu}}^{\frac{1}{q_{0}^{*}}} u(E_{P})^{\frac{2}{p}} v(E_{P})^{\frac{1}{p^{*}}}.
\end{split}\end{align*}
After noting that $\frac{2}{p_{0}(r - 1)} = \frac{2 \gamma(p)}{(q_{0}/p)'}$
the proof of \eqref{e.strongsecondmain} then proceeds in an identical manner to the case $p \geq \mathfrak{p}$.
 \qed

\vspace*{0.2in}

\subsection{Proof of Corollary \ref{cor:Weighted}}
We start by noting that
\begin{align*}\begin{split}  
 \norm{S f}_{L^{\mathfrak{p}}(w)}^{2} &= \norm{(S
   f)^{2}}_{L^{\frac{\mathfrak{p}}{2}}(w)} \\
 &= \sup_{\norm{g}_{L^{\mathfrak{p}^{*}}(\sigma)} = 1}
 \abs{\langle (Sf)^{2}, g \rangle},
\end{split}\end{align*}
where $\sigma := w^{1 - \br{\frac{\mathfrak{p}}{2}}'} = w^{1 - \mathfrak{p}^{*}}$ is the
$A_{\frac{\mathfrak{p}}{2}}$-conjugate weight of $w$. Thus, in order
to prove the desired result, it is sufficient to demonstrate the
estimate
\begin{equation}
  \label{eqtn:Bounded1}
  \abs{\langle (S f)^{2}, g \rangle} \lesssim  \brs{w^{(q_{0}/\mathfrak{p})'}}_{A_{\phi(\mathfrak{p})}}^{\frac{1}{q_{0}^{*}}}
  \norm{f}_{L^{\mathfrak{p}}(w)}^{2} \norm{g}_{L^{\mathfrak{p}}(\sigma)}.
\end{equation}
For the critical index $\mathfrak{p}$ this is an easy consequence of
Theorem \ref{t.main}, estimate \eqref{e.strongsecondmain} and a density argument.

Applying the sharp restricted range extrapolation (\cref{thm:Extrapolation})
yields that for any $p \in (p_{0},q_{0})$ and weight $w \in A_{\frac{p}{p_{0}}} \cap RH_{(\frac{q_{0}}{p})'}$,
\begin{equation}
  \label{eqtn:Cor1.8Beta}
  \norm{S f}_{L^{p}(w)} \lesssim
  \brs{w^{(\frac{q_{0}}{p})'}}^{\beta(p,\mathfrak{p})/(2 q_{0}^{*})}_{A_{\phi(p)}}
\end{equation}
where $\beta(p,\mathfrak{p}) = \max\br{1, \frac{(q_{0} - p)(\mathfrak{p} - p_{0})}{(q_{0} - \mathfrak{p})(p - p_{0})}}$.

We check that this matches the power $\gamma(p)$ in \cref{cor:Weighted}.
Let $\omega(p) \coloneqq (q_{0} - p)/(p - p_{0})$ for $p\in (p_0,q_0)$.
Then $\beta(p,\mathfrak{p}) = \max( 1, \omega(p) /\omega(\mathfrak{p}) )$.
  Since $\omega(p)$ is decreasing in $p$ and $\beta(\mathfrak{p},\mathfrak{p}) = 1$,
  then $\beta(p,\mathfrak{p}) = 1$ for $p \in [\mathfrak{p},q_0)$.
  In this range of exponent we have
  \begin{equation*}
    \norm{S f}_{L^{p}(w)} \lesssim
    \brs{w^{(\frac{q_{0}}{p})'}}^{1/(2 q_{0}^{*})}_{A_{\phi(p)}}
    \lesssim \br{\brs{w}_{A_{\frac{p}{p_{0}}}} \cdot
      \brs{w}_{RH_{(\frac{q_{0}}{p})'}}}^{(\frac{q_{0}}{p})'/(2 q_{0}^{*})},
  \end{equation*}
  where the last inequality is the bound on the weight characteristic
  given in \eqref{eqtn:WeightProperty}.

When $p < \mathfrak{p}$, instead $\beta(p,\mathfrak{p}) = \omega(p)/\omega(\mathfrak{p})$.
Using the identity \eqref{eqtn:Critical} for $\mathfrak{p}$,
one can see that $\omega(\mathfrak{p}) \cdot 2q_0^* = q_0$.
This immediately gives
\begin{equation*}
  \frac{\beta(p,\mathfrak{p})}{2q_0^*} = \frac{\omega(p)}{\omega(\mathfrak{p}) \cdot 2q_0^*} = \frac{\omega(p)}{q_0} = \frac{1}{\br{q_0/p}'} \frac{1}{p-p_0}.
\end{equation*}
Then \eqref{eqtn:Cor1.8Beta} followed by \eqref{eqtn:WeightProperty} implies that for $p \in (p_0,\mathfrak{p})$
\begin{equation*}
  \norm{S f}_{L^{p}(w)} \lesssim
  \brs{w^{(\frac{q_{0}}{p})'}}^{\frac{1}{\br{q_0/p}'} \frac{1}{p-p_0}}_{A_{\phi(p)}}
  \lesssim \br{\brs{w}_{A_{\frac{p}{p_{0}}}} \cdot
    \brs{w}_{RH_{(\frac{q_{0}}{p})'}}}^{1/(p-p_0)}
\end{equation*}
The exponent in the above inequality matches the hypothesised exponent of
\eqref{eqtn:WeightProperty}, allowing us to conclude our proof. \qed

\section{Sharpness of the sparse form for
  \texorpdfstring{$p>2$}{p>2}}
\label{sec:Sharpness}
In this section we will use the notation $\sim$ to indicate asymptotic behaviour
and we will work in $\mathbb{R}$ with the Lebesgue measure. 
The sharpness in \cref{thm:weight_bounds} is a consequence of the following proposition.
The proof, although different, follows the reasoning in \cite[\S 7]{bernicot2016sharp}.

\begin{proposition}
  For $p\in (2,q_0)$, there exists a sparse collection $\mathcal S$ and for every $0<\epsilon<1$,
  there exist sequences of functions $f_\epsilon$ and $g_\epsilon$ and weights $w_\epsilon$ such that
  \begin{equation}\label{asymp}
    \sum_{P \in \mathcal{S}} \left( \fint_P \abs{f_\epsilon}^{p_0} \D{x}\right)^{2/p_0} \left( \fint_P \abs{g_\epsilon}^{q_{0}^{*}} \D{x} \right)^{1/q_{0}^{*}} \abs{P}
    \sim \br{\brs{w_\epsilon}_{A_{\frac{p}{p_{0}}}} \cdot \brs{w_\epsilon}_{RH_{(\frac{q_{0}}{p})'}}}^{2\gamma(p)} \norm{f_\epsilon}_{L^{p}(w_\epsilon)}^{2} \norm{g_{\epsilon}}_{L^{(p/2)'}(\sigma_\epsilon)}
  \end{equation}
  as $\epsilon \to 0$, where  
  \begin{equation*}
    \gamma(p) \coloneqq \max \br{\frac{1}{p - p_{0}}, \br{\frac{q_{0}}{p}}' \frac{1}{2q_{0}^{*}}} \,\,\, \text{ and } \,\,\, \sigma_\epsilon\coloneqq w_{\epsilon}^{1-(p/2)'}.
  \end{equation*}
\end{proposition}

\begin{proof}
The proof is divided into two cases, the case where $p\leq \mathfrak p$ and the case where $p\geq \mathfrak p$. In both of them, the sparse collection considered is $\mathcal S=\{I_{n}\coloneqq[0,2^{-n}]\, :\, \text{for } n\in \mathbb{N}\}$.

For $2<p\leq  \mathfrak p$ and fixed $0<\epsilon<1$, consider the functions
\begin{align*}
  f_\epsilon(x) &\coloneqq x^{-\frac{1}{p_0}+\epsilon}\chi_{[0,1]}\\
  g_\epsilon(x) &\coloneqq x^{-\frac{1}{p_0^{*}}+\epsilon}\chi_{[0,1]}\\
  w_\epsilon(x) &\coloneqq |x|^{\frac{p}{p_0}-1-\epsilon}\chi_{[0,1]}\\
  \sigma_{\epsilon}(x)&= |x|^{\left(\frac{p}{p_0}-1-\epsilon\right)(1-p^*)}\chi_{[0,1]},
\end{align*}
where $  \sigma_{\epsilon}$ is the dual weight to $  w_\epsilon$ in $A_{p/2}$.

\vspace{0.1in}

Then 
\begin{align*}
  \left( \fint_P \abs{f_\epsilon}^{p_0} \D{x} \right)^{2/p_0} & =\left(\frac{2^{\frac{n}{p_0}-n\epsilon}}{\left ( p_0\epsilon\right)^{1/p_0} }\right)^2 \sim \epsilon^{-2/p_0}2^{-2n\epsilon}2^{2n/p_0}\\
  \left( \fint_P \abs{g_\epsilon}^{q_{0}^{*}} \D{x} \right)^{1/q_{0}^{*}} & \sim \frac{2^{\frac{n}{p_0^{*}}-n\epsilon}}{\left(1-\frac{q_{0}^{*}}{p_{0}^{*}}+q_{0}^{*}\epsilon \right)^{1/q_{0}^{*}}}\sim 2^{-n\epsilon} 2^{n/p_0^{*}}.
\end{align*}
as $\epsilon\to 0$. The left hand side of \eqref{asymp} follows the asymptotic behaviour
\begin{equation*}%
 \sum_{P \in \mathcal{S}} \left( \fint_P \abs{f_\epsilon}^{p_0} \D{x} \right)^{2/p_0} \left( \fint_P \abs{g_\epsilon}^{q_{0}^{*}} \D{x} \right)^{1/q_{0}^{*}} \abs{P} \sim \epsilon^{-2/p_0}\sum_{n=0}^{\infty}2^{-3n\epsilon}\sim \epsilon^{-2/p_0} \epsilon^{-1}.
\end{equation*}

For power weights, the asymptotics of the $A_p$ and $RH_q$ are well understood,
see for instance \cite{MR1124164}. Therefore, as $\epsilon \to 0$, we have

\begin{align*}
  &\brs{w_\epsilon}_{A_{\frac{p}{p_{0}}}} \sim \epsilon^{-\left(\frac{p}{p_0}-1\right)},\\
  &\brs{w_\epsilon}_{RH_{(\frac{q_{0}}{p})'}} \sim 1,
\end{align*}

Moreover we compute

\begin{align*}
\norm{f_\epsilon}_{L^{p}(w_\epsilon)} &= \left(\int_{0}^{1} x^{\frac{-p}{p_0}+p\epsilon} x^{\frac{p}{p_0}-1-\epsilon} \D{x}\right)^{1/p}\\
  &=\left(\int_{0}^{1} x^{-1+(p-1)\epsilon} \D{x}\right)^{1/p}\sim \epsilon^{-1/p},\\
\end{align*}
as $\epsilon \to 0$. 

\begin{align*}
  \norm{g_{\epsilon}}_{L^{(p/2)'}(\sigma_\epsilon)}&= \left(\int_{0}^{1} x^{\frac{-p^{*}}{p_0^{*}}+p^{*}\epsilon} x^{\left(\frac{p}{p_0}-1-\epsilon\right) \left( 1-p^{*}\right)}  \D{x}\right)^{1/p^{*}} \\
                                               &= \left(\int_{0}^{1} x^{-1+(2p^{*}-1)\epsilon-\frac{p^{*}}{p_0^{*}}+\frac{p}{p_0}-\frac{pp^{*}}{p_0}+p^{*}}  \D{x}\right)^{1/p^{*}}.
\end{align*}
Using the definition of $p^{*}$ and $p_{0}^{*}$, we note $ -\frac{p^{*}}{p_0^{*}}+\frac{p}{p_0}-\frac{pp^{*}}{p_0}+p^{*}=0$, therefore
$$
\norm{g_{\epsilon}}_{L^{(p/2)'}(\sigma_\epsilon)}= \left(\int_{0}^{1} x^{-1+(2p^{*}-1)\epsilon} \D{x}\right)^{1/p^{*}}\sim \epsilon^{-1/p^{*}},
$$
as $\epsilon \to 0$. 

We conclude that the right hand side of \eqref{asymp}
behaves as $\epsilon^{-\left(\frac{p}{p_0}-1\right)\left(\frac{2}{p-p_0}\right)}\epsilon^{-2/p} \epsilon^{-1/p^{*}}=\epsilon^{-1}\epsilon^{-2/p_0}$ as $\epsilon \to 0$,
which is exactly the asymptotic behaviour of the left hand side of \eqref{asymp} as desired.

\vspace{0.2in}

\noindent

For $\mathfrak p\leq p<q_0 $ and fixed $0<\epsilon<1$, consider the functions
\begin{align*}
  f_\epsilon(x) &\coloneqq x^{-\frac{1}{q_0}+\epsilon}\chi_{[0,1]}\\
  g_\epsilon(x) &\coloneqq x^{-\frac{1}{q_0^{*}}+\epsilon}\chi_{[0,1]}\\
  \sigma_\epsilon(x) &\coloneqq |x|^{\frac{p^*}{q_0^*}-1-\epsilon}\chi_{[0,1]}\\
  w_\epsilon &= |x|^{\left(\frac{p^*}{q_0^*}-1-\epsilon\right)\left(1-p/2\right)},
\end{align*}
where $  \sigma_{\epsilon}$ is the dual weight to $  w_\epsilon$ in $A_{p/2}$.

\vspace{0.1in}

Then 
\begin{align*}
  \left( \fint_P \abs{g_\epsilon}^{q_{0}^{*}} \D{x} \right)^{1/q_{0}^{*}} & = \frac{2^{\frac{n}{q_0^*}-n\epsilon}}{\left ( q_0^{*}\epsilon\right)^{1/q_0^{*}} }\sim \epsilon^{-1/ q_0^{*}}2^{-n\epsilon}2^{n/ q_0^{*}}\\
  \left( \fint_P \abs{f_\epsilon}^{p_0} \D{x} \right)^{1/p_0}  & = \frac{2^{\frac{n}{q_0}-n\epsilon}}{\left(1-\frac{p_0}{q_0}+p_0 \epsilon \right)^{1/p_{0}}}\sim 2^{-n\epsilon} 2^{n/q_0}.
\end{align*}
as $\epsilon\to 0$. And the right hand side of \eqref{asymp} follows the asymptotic behaviour
\begin{align*}%
  \sum_{P \in \mathcal{S}}  \left ( \fint_{P} \abs{f_\epsilon}^{p_0}\D{x} \right )^{2/p_0}\left ( \fint_{P}\abs{g_\epsilon}^{q_0^{*}} \D{x} \right )^{1/q_0^{*}} \abs{P}
  &\sim \epsilon^{-1/q_0^{*}}\sum_{n=0}^{\infty} 2^{-n} 2^{\frac{2n}{q_0}-2\epsilon n} 2^{-n\epsilon+ \frac{n}{q_0^{*}}} \\
  &\sim\epsilon^{-1/q_0^{*}}\sum_{n=0}^{\infty}2^{-3n\epsilon}\sim \epsilon^{-1/q_0^{*}} \epsilon^{-1}.
\end{align*}

For the power weights $w_\epsilon$, as $\epsilon \to 0$, we have,
\begin{align*}
&\brs{w_\epsilon}_{A_{\frac{p}{p_{0}}}} \sim 1,\\
&\brs{w_\epsilon}_{RH_{(\frac{q_{0}}{p})'}} \sim \epsilon^{\frac{-1}{(q_0/p)'}}.
\end{align*}

Moreover we compute

$$
\norm{g_{\epsilon}}_{L^{(p/2)'}(\sigma_\epsilon)}=\left(\int_{0}^{1} x^{-1+(p^*-1)\epsilon} \D{x}\right)^{1/p^{*}}\sim \epsilon^{-1/p^{*}},
$$
as $\epsilon \to 0$, and

\begin{align*}
\norm{f_\epsilon}_{L^{p}(w_\epsilon)}& = \left(\int_{0}^{1}x^{\frac{-p}{q_0}+\epsilon p}x^{\left( \frac{p^{*}}{q_0^{*}}-1-\epsilon\right)\left(1-\frac{p}{2} \right)}\D{x}\right) ^{1/p}\\
&  = \left(\int_{0}^{1}x^{\left(\frac{3p}{2}-1\right)\epsilon-1 -\frac{p}{q_0} +\frac{p^{*}}{q_0^{*}}-\frac{pp^{*}}{2q_0^{*}} +\frac{p}{2}  }\D{x}\right) ^{1/p}.\\
\end{align*}

Using the definition of $p^{*}$ and $q_{0}^{*}$, we note $ -\frac{p}{q_0}+\frac{p^{*}}{q_0^{*}}-\frac{pp^{*}}{2q_0^{*}}+\frac{p}{2}=0$, therefore

$$
\norm{f_\epsilon}_{L^{p}(w_\epsilon)} =\left(\int_{0}^{1} x^{(3p/2-1)\epsilon-1} \D{x}\right)^{1/p}\sim \epsilon^{-1/p}.
$$

We conclude the right hand side of \eqref{asymp} behaves as $\epsilon^{-1/q_0^*}\epsilon^{-2/p} \epsilon^{-1/p^{*}}=\epsilon^{-1}\epsilon^{-1/q_0^*}$ as $\epsilon \to 0$, which is exactly the asymptotic for the left hand side of \eqref{asymp} as desired.
\end{proof}

\subsection{Upper bound on asymptotic behaviour}
\label{subsec:asymptotic_behaviour}

In this section we discuss the connection between
sharp weighted estimates for an operator $T$
and the asymptotic behaviour of its unweighted norm $\lVert T \rVert_{L^p\to L^p}$.
We recall the definition of $\gamma(q_0)$ from \cite[Definition 5.1]{FreyBas}.
Let $T$ be a bounded operator on $L^p$ for $p \in (p_0,q_0)$.
\begin{deff}
  For $q_0 < \infty$ define
  \begin{equation*}
    \gamma(q_0) \coloneqq \sup \big\{\gamma \ge 0 \,\vert\, \forall \epsilon > 0, \limsup_{p \to q_0}(q_0 - p)^{\gamma - \epsilon} \lVert T \rVert_{L^p\to L^p} = \infty \big\} ,
  \end{equation*}
  and for $q_0 = \infty$
  \begin{equation*}
    \gamma(\infty) \coloneqq \sup\big\{ \gamma \ge 0 \,\vert\, \forall \epsilon >0, \limsup_{p\to\infty} \frac{\lVert T \rVert_{L^p\to L^p}}{p^{\gamma-\epsilon}} = \infty \big\}.
  \end{equation*}
\end{deff}

We say that an operator $T$ admits
a $(p_0,q_0)$ quadratic sparse domination
if it satisfies a bound as the one in \cref{t.main}.
We have the following upper bound
on the unweighted norm of $T$.
\begin{prop}
  Let $q^* \coloneqq (q/2)'$.
  If $T$ admits a $(p_0,q_0)$ quadratic sparse domination then
  for $p>2$ we have
  \begin{equation*}
    \lVert T \rVert_{L^p\to L^p} \lesssim \left[ \left(\frac{p}{p_0}\right)'\right]^{\frac1{p_0}} \left[ \left(\frac{p^*}{q_0^*}\right)'\right]^{\frac12 \frac1{q_0^*}}
  \end{equation*}
  and in particular
  \begin{equation}\label{eq:gamma_upper_bound}
    \gamma(q_0) \le \frac1{2{q_0^*}}. %
  \end{equation}
\end{prop}
\begin{proof}
  As in \cite[Remark 3.4]{FreyBas},
  let $\mathcal{S}$ be a $\eta$-sparse family.
  For $p>2$ we have
  \begin{align*}
    \sum_{P \in \mathcal{S}} \left(\fint_P \lvert f\rvert^{p_0} \D{\mu}\right)^{2/p_0} \left(\fint_P \lvert g\rvert^{q_0^*} \D{\mu} \right)^{1/q_0^*} \lvert P \rvert
    & \lesssim \frac1{\eta} \lVert \dyadicM_{\frac{p_{0}}2}( \lvert f\rvert^2) \rVert_{L^{p/2}} \lVert \dyadicM_{q_0^*} g \rVert_{L^{(p/2)'}} \\
    & \lesssim \frac1{\eta}  \left[ \left(\frac{p}{p_0}\right)'\right]^{\frac2{p_0}}
      \left[\left(\frac{p^*}{q_0^*}\right)'\right]^{\frac1{q_0^*}}
      \lVert f \rVert_{L^{p}}^2 \lVert  g \rVert_{L^{p^*}} 
  \end{align*}
  where the last inequality follows from
  the bound on the $L^p$-norm of $\dyadicM$ in \eqref{eq:bound_M_w}, since
  \begin{align*}
      \lVert \dyadicM_{\frac{p_0}2}(\lvert f\rvert^2 )\rVert_{L^{p/2}} = \lVert \dyadicM (\lvert f\rvert^{p_0}) \rVert_{L^{p/p_0}}^{2/p_0} .
  \end{align*}

\end{proof}

\begin{rmk}  
  The upper bound on $\gamma(q_0)$ in \eqref{eq:gamma_upper_bound}
  implies that, if $\gamma(q_0)$ equals $1/(2q_0^*)$
  then the weighted estimates \eqref{eqtn:Bounded} are sharp.

\end{rmk}

\end{document}